\documentclass[a4paper,10pt,reqno]{amsart}
\usepackage[ams]{optional}

\usepackage{amssymb, amsmath, amsthm, amsfonts}
\opt{ams}{\usepackage{amsthm}}
\usepackage{graphicx, psfrag, gensymb}
\usepackage{bbm}
\usepackage{graphicx,mathrsfs}
\usepackage{stmaryrd} % for \mapsfrom etc
\usepackage{array,blkarray}
\usepackage{xypic} 
\usepackage{longtable}
\xyoption{all}

\usepackage{fix-cm}
\usepackage{comment}

\usepackage[colorlinks,hyperfootnotes=false]{hyperref} % comment out for arXiv?

\usepackage{float}
\usepackage{setspace} 

\newcolumntype{C}[1]{>{\centering\arraybackslash$}p{#1}<{$}}

\usepackage{tikz,tikz-cd,braids}
\usetikzlibrary{matrix,arrows,decorations.pathmorphing,decorations.pathreplacing,calc}
\tikzset{commutative diagrams/diagrams={baseline=-2.5pt},commutative diagrams/arrow style=tikz}
\tikzset{
        vertex/.style={circle,fill=black,inner sep=1pt,outer sep=3pt},
        dotarrow/.style={dotted,->},
        triangulation/.style={line width=0.5pt,dotted},
        gap/.style={inner sep=0.5pt,fill=white}
        }
\def\sepForEquals{0.2ex}
\opt{ams}{
\addtolength{\oddsidemargin}{-.6in}
\addtolength{\evensidemargin}{-.6in}
\addtolength{\textwidth}{1.2in}
\addtolength{\marginparwidth}{-6.in}

\addtolength{\topmargin}{-.3in}
\addtolength{\textheight}{0.6in}

}

\newcommand{\marginparstretch}{0.8}
\let\oldmarginpar\marginpar
\renewcommand\marginpar[1]{\oldmarginpar[\framebox{\setstretch{\marginparstretch}\begin{minipage}{\marginparwidth}{\raggedleft\footnotesize #1}\end{minipage}}]{\framebox{\setstretch{\marginparstretch}\begin{minipage}{\marginparwidth}{\raggedright\footnotesize #1}\end{minipage}}}}

\usepackage{subcaption}

\newcommand{\chapter}

\def\tvdots{\raisebox{4pt}{$\scalebox{0.75}{\vdots}$}}

\newcommand{\quotes}[1]{\textquoteleft{#1}'}
\newcommand{\set}[1]{\{{#1}\}}
\newcommand{\cone}[1]{\ensuremath{\operatorname{Cone}\left({#1}\right)}}
\newcommand{\genconds}[2]{\ensuremath{\left\langle\begin{array}{c|c} #1 & #2 \end{array}}\right\rangle}
\newcommand{\setconds}[2]{\ensuremath{\left\{\begin{array}{c|c} #1 & #2 \end{array}}\right\}}

\newcommand{\quotstack}[2]{\ensuremath{\left[\,#1\,/\,#2\,\right]}}
\newcommand{\quotstackinline}[2]{\ensuremath{[\,#1\,/\,#2\,]}}
\newcommand{\quotstackBig}[2]{\ensuremath{\Big[\,#1\,/\,#2\,\Big]}}
\def\Db{\mathop{\rm{D}^b}\nolimits}

\newcommand\xyhook{\ar@{^{(}->}}

\newcommand\Z{\mathbb Z}
\newcommand\C{\mathbb C}

\newcommand\isoto{\stackrel{\sim}{\To}}

\newcommand\acts{\curvearrowright}
\newcommand\into{\hookrightarrow}

\newcommand\To{\longrightarrow}

\newcommand\Hom{\operatorname{Hom}}

\renewcommand\Im{\operatorname{Im}}

\renewcommand\P{\mathbb P}

\newcommand\pt{\operatorname{pt}}

\makeatletter \@addtoreset{equation}{section} \makeatother

\DeclareCaptionSubType*{figure}

\newcommand{\pgap}{\vspace{10pt}}

\theoremstyle{plain}
\newtheorem{prop}[equation]{Proposition}
\newtheorem{thm}[equation]{Theorem}
\newtheorem{lem}[equation]{Lemma}
\newtheorem{cor}[equation]{Corollary}
\newtheorem{keythm}{Theorem}

\theoremstyle{remark}
\newtheorem{rem}[equation]{Remark}
\newtheorem*{acks}{Acknowledgements}

\theoremstyle{definition}

\newtheorem{notn}[equation]{Notation}
\newtheorem{eg}[equation]{Example}

\newcommand{\cE}{\mathcal{E}}
\newcommand{\cF}{\mathcal{F}}
\newcommand{\cG}{\mathcal{G}}
\newcommand{\cO}{\mathcal{O}}
\newcommand{\cS}{\mathcal{S}}
\newcommand{\cV}{\mathcal{V}}
\newcommand{\cX}{\mathcal{X}}
\newcommand{\cY}{\mathcal{Y}}
\newcommand{\cW}{\mathcal{W}}
\newcommand{\cN}{\mathcal{N}}

\renewcommand{\Re}{\operatorname{Re}}

\newcommand{\Res}{\operatorname{Res}}

\newcommand{\bijar}[1][]{%
 \ar[#1]
 \ar@<0.7ex>@{}[#1]|-*=0[@]{\sim}} 
\newcommand{\bijarswap}[1][]{%
 \ar[#1]
 \ar@<0.7ex>@{}[#1]|-*=0[@]{\sim}} 

\mathchardef\-="2D

\newcommand{\fin}{\scriptscriptstyle{\mathrm{fin}}}
\newcommand{\crs}{\scriptscriptstyle{\mathrm{crs}}}
\newcommand{\free}{\scriptscriptstyle{\mathrm{free}}}

\newcommand{\dualCoord}[1]{c_{#1}}
\newcommand{\polyVec}[2]{\tilde{\alpha}_{#1}({#2})}
\newcommand{\polyVecSymbol}[2]{e_{#1}({#2})}
\newcommand{\polyBaseVec}[1]{\beta_{#1}}
\newcommand{\polyTopVec}[1]{\alpha_{#1}}
\newcommand{\simplex}[1]{[{#1}]}
\newcommand{\Catni}{\mathbf{Cat}_{1}}

\def\angleA{-10}
\def\angleB{35}
\def\scaleA{1.1}
\def\scaleB{0.9}
\def\scaleC{1.2}

\newcommand{\polytope}[9]{
\def\annotate{#7}
\def\yes{yes}
\draw[line width=1.5pt]
	($(#1:#2)+(#5,#6)$) coordinate(#8B1) node[right]{\ifx\annotate\yes\scriptsize $\polyBaseVec{1}$\fi} --
	(#5,#6) coordinate(#8B0) node[left]{\ifx\annotate\yes\scriptsize $\polyBaseVec{0}$\fi} -- 
	($(#3:#4)+(#5,#6)$) coordinate(#8B2) node[below]{} -- cycle; % top
\draw[line width=1.5pt]
	($(#1:#2)+(#5,#6+#9)$) coordinate(#8A1) node[right]{\ifx\annotate\yes\scriptsize $\polyTopVec{1}$\fi} --
	(#5,#6+#9) coordinate(#8A0) node[left]{\ifx\annotate\yes\scriptsize $\polyTopVec{0}$\fi} -- 
	($(#3:#4)+(#5,#6+#9)$) coordinate(#8A2) node[above]{\ifx\annotate\yes\scriptsize $\polyTopVec{2}$\fi} -- cycle;
\draw[line width=1.5pt] (#8B0) -- +(0,#9);
\draw[line width=1.5pt] (#8B1) -- +(0,#9);
\draw[line width=1.5pt] (#8B2) -- +(0,#9);

\coordinate (#8Centre) at ($(#5,#6)+0.33*(#1:#2)+0.33*(#3:#4)+(0,0.5)$);
}

\newcommand{\polytopeStdTriangulation}[9]{

\polytope{#1}{#2}{#3}{#4}{#5}{#6}{#7}{#8}{#9}

\draw[triangulation] (#8B1) -- (#8A0);
\draw[triangulation] (#8B2) -- (#8A0);
\draw[triangulation] (#8B2) -- (#8A1);

\def\labelsep{0.1}
\coordinate (#8L0) at ($0.5*#9*(#8B0)+0.5*#9*(#8A0)-(\labelsep,0)$);
\coordinate (#8L1) at ($0.5*#9*(#8B1)+0.5*#9*(#8A1)-(\labelsep,0)$);
\coordinate (#8L2) at ($0.5*#9*(#8B2)+0.5*#9*(#8A2)+(\labelsep,0)$);

\def\arrowlen{0.5}
\draw[->] ($(#8L0)+(\angleA:-\arrowlen)$) node[inner sep=0.5pt,fill=white] {\scriptsize $\Delta_0$} -- (#8L0);
\draw[->] ($(#8L1)+(\angleA:\arrowlen)$) node[inner sep=0.5pt,fill=white] {\scriptsize $\Delta_1$} -- (#8L1);
\draw[->] ($(#8L2)+(\angleB:\arrowlen)$) node[inner sep=0.5pt,fill=white] {\scriptsize $\Delta_2$} -- (#8L2);

}

\newcommand\quiver[9]{ % b1, a1, b0, a0, b2, a2

\def\basic{basic}

\node(#9V1) at (#7,#8) [vertex] {};
\ifx#9\basic \node[above of=#9V1] {\scriptsize$1$} \fi;
\node(#9V2) at ($(#9V1)+(1.6,0)$) [vertex] {};
\ifx#9\basic \node[above of=#9V2] {\scriptsize$2$} \fi;
\node(#9V0) at ($(#9V1)+(0.8,-1)$) [vertex] {};
\ifx#9\basic \node[below of=#9V0] {\scriptsize$0$} \fi;

\foreach \i/\j/\bshift/\ashift/\bpos/\apos/\barrow/\aarrow in {
	2/1/.3/-.3///#1/#2,
	1/0/-.5/.5///#3/#4,
	0/2/-.5/.5///#5/#6} {
	\draw[\barrow,bend left=15,looseness=1] (#9V\i) to node[gap, \bpos ,pos=0.5]{\scriptsize$b_\j$} (#9V\j);
	\draw[\aarrow,bend left=15,looseness=1] (#9V\j) to node[gap, \apos ,pos=0.5]{\scriptsize$a_\j$} (#9V\i);
	};
	
	\coordinate (#9Centre) at ($(#9V1)+(0.8,-0.5)$)

}

\newcommand{\defBGColor}{gray!20}
\newcommand{\defCurveCoords}[2]{
 (\halfwidth*-1*#1+\pinch*#1-#1*#2,#2) .. controls (-1/2*\halfwidth*#1-#1*#2,3/4*\intersection+#2) and (-1/4*\halfwidth*#1-#1*#2,\intersection+#2) .. (\crossing*#1-#1*#2,\intersection+\crossing+#2)
}
\newcommand{\defCurveCoordsRev}[2]{
(\crossing*#1-#1*#2,\intersection+\crossing+#2) .. controls (-1/4*\halfwidth*#1-#1*#2,\intersection+#2) and (-1/2*\halfwidth*#1-#1*#2,3/4*\intersection+#2) .. (\halfwidth*-1*#1+\pinch*#1-#1*#2,#2)
}

\begin{document}

\title[Mixed braid group actions from surface singularities]{Mixed braid group actions from deformations of surface singularities}%
\author{Will Donovan}
\address{Will Donovan, The Maxwell Institute, School of Mathematics, James Clerk Maxwell Building, The King's Buildings, Mayfield Road, Edinburgh, EH9 3JZ, U.K.}
\email{Will.Donovan@ed.ac.uk}
\author{Ed Segal}
\address{Ed Segal, Department of Mathematics, Imperial College London, London, SW7 2AZ, U.K.}
\email{edward.segal04@imperial.ac.uk}

\begin{abstract}
We consider a set of toric Calabi--Yau varieties which arise as deformations of the small resolutions of type $A$ surface singularities. By careful analysis of the heuristics of B-brane transport in the associated GLSMs, we predict the existence of a mixed braid group action on the derived category of each variety, and then prove that this action does indeed exist. This generalizes the braid group action found by Seidel and Thomas for the undeformed resolutions. We also show that the actions for different deformations are related, in a way that is predicted by the physical heuristics.
\end{abstract}
\subjclass[2010]{Primary 14F05, 18E30; Secondary 14J33, 20F36}
% 14F05 Sheaves, derived categories of sheaves and related constructions
% 	18E30 Derived categories, triangulated categories
%	14J33	Mirror symmetry
%	20F36 Braid groups; Artin groups

\thanks{W.D. is grateful for the support of EPSRC grant EP/G007632/1, and E.S. for the support of an Imperial College Junior Research Fellowship.}
\maketitle

\setcounter{tocdepth}{2}
\tableofcontents
% ----------------------------------------------------------------

\section{Introduction}\label{sect.introduction}

Let $V$ be a vector space, and let $T$ be a torus acting on $V$. If we pick a character $\theta$ of $T$, we have a stability condition for the associated GIT problem, and can form  a GIT quotient $X_\theta = V \sslash_\theta T$. We'll borrow some physics terminology and refer to these different possible GIT quotients as \textit{phases}.

There is now a well-developed general theory \cite{DHL, BFK}, inspired by the physics paper \cite{HHP} but going back to ideas of Kawamata, that allows us to compare the derived categories of different phases. In particular, if $T$ acts through $\operatorname{SL}(V)$, so that all the phases $X_\theta$ are Calabi--Yau, then all their derived categories $\Db(X_\theta)$ will be equivalent. These equivalences are not canonical: whenever we cross a wall in the space of stability conditions, we have a countably infinite set of equivalences between the two phases that lie on either side.  Consequently we can produce autoequivalences of the derived category of a single phase  $X_\theta$: we pass through an equivalence to the derived category of a different phase, then we pass back again, using a different equivalence.  In this way we can produce a whole group of autoequivalences acting on $\Db(X_\theta)$. Then we can ask the question: what group is it?

 To get a prediction, we can turn to quantum field theory. The data of $T$ acting on $V$ determines a Gauged Linear Sigma Model, a particular kind of supersymmetric $2$-dimensional QFT. The model has a parameter, called the (complexified) Fayet--Iliopoulos parameter, which takes values in a  space which we'll call $\cF$.   This parameter space $\cF$ is a complex manifold (or orbifold) locally modelled on the dual Lie algebra $\mathfrak{t}^*$ of $T$, but it can have non-trivial global structure. In particular limiting regions in $\cF$, the theory is expected to reduce to a sigma model with target one of the phases $X_\theta$, with different phases occurring at different limits.  

There is also believed to be a triangulated category associated to the GLSM, called the category of B-branes. If we assume the Calabi--Yau condition then this category is independent of the FI parameter, but only up to isomorphism. We should think of it as a `local system' of categories over the space $\cF$. When we approach one of the limiting regions, the category of B-branes becomes identified (not quite canonically\footnote{The ambiguity is tensoring by line bundles.}) with the derived category $\Db(X_\theta)$ of the corresponding phase. If we travel along a path in $\cF$ from one limit to another then we can perform `parallel transport' of the B-branes, and produce a derived equivalence between the corresponding phases. Similarly if we travel around a large loop in $\cF$ then we get a monodromy autoequivalence, acting on the derived category of a single phase. The derived equivalences and autoequivalences that arise from this `brane transport' are believed to be exactly the ones that arise from variation of GIT quotient in the mathematical constructions mentioned above. So this picture tells us which group we should expect to see acting on $\Db(X_\theta)$: it's the fundamental group of $\cF$.\footnote{This picture doesn't tell us whether or not the action is faithful, however.}

 Providing a  rigorous definition of this QFT is an immensely difficult problem which will probably not be resolved for many years, and consequently an intrinsic definition of the FI parameter space $\cF$ is not known. Fortunately, what is known is a completely precise heuristic recipe that tells us how to construct $\cF$. The justification for this recipe (at least mathematically) is toric mirror symmetry, since $\cF$ is related to the complex moduli space of the mirror theory.\footnote{The fact that $T$ is a torus is crucial here, if we replace it with a non-abelian group then as far as we know no recipe for $\cF$ exists.}

If we apply this recipe, and  can compute $\pi_1(\cF)$, then our heuristics will have given us a precise prediction about derived autoequivalences which we can then attempt to prove. In this paper we're going to carry out this program, and prove the prediction, for a particular set of examples.
\pgap

Our examples arise from a well-known piece of geometry, namely resolutions of the $A_k$ surface singularities, and their deformations. The resulting set of GLSMs has another important feature: certain phases of each model embed, in a natural way, into particular phases of some of the other models. In other words, we're considering some toric varieties, and also some torically-embedded subvarieties. 
 
When two GLSMs are related in this way then a rough physical argument suggests that there should be a map from the FI parameter space for the ambient variety to the FI parameter space for the subvariety. We don't know a mathematical justification for this argument, but in our examples we see that these maps do indeed exist, and in fact they're covering maps. Consequently the fundamental groups of the two FI parameter spaces are related, and this suggests that the brane transport autoequivalences for the two models should also be related. We show that these predicted relationships between the examples do indeed hold.

\subsection{Main result}

We now briefly state our main theorem, leaving some of the details to Section~\ref{sect.examples_results}. Firstly, let $Y$ denote a minimal resolution of a local model for an $A_k$ surface singularity, as described explicitly in Section~\ref{section.statement}. In a seminal paper \cite{ST}, Seidel and Thomas construct a faithful action $B_{k+1} \acts \Db(Y)$ of the braid group on $k+1$ strands on the derived category of $Y$, by spherical twists. Now given a partition $\Gamma$ of the set of strands, we may define:

\begin{enumerate}
\item A {\em mixed braid group} $B_\Gamma$, namely that subgroup of the braid group $B_{k+1}$ consisting of braids which respect the partition $\Gamma$ (see \eqref{eqn.mixed_braid_defn}). 
\item A certain deformation $X_\Gamma$ of the surface $Y$ (see Section~\ref{section.construction}). This deformation has $s$ parameters, where $s+1$ is the number of pieces in the partition.
\end{enumerate}

\begin{keythm}[Corollary~\ref{cor.maincor}]\label{keythm} For each partition $\Gamma$ of $\{0,\ldots,k\}$, there is a faithful action of the mixed braid group $B_{\Gamma} \acts \Db(X_\Gamma)$ on the derived category of the deformation $X_\Gamma$.
\end{keythm}

When $\Gamma$ has only one part, the theorem recovers the original Seidel--Thomas braid group action on~$Y$: in this case, the usual braid generators $t_i$ act by  spherical twists. When $\Gamma$ is the finest partition, with $k+1$ parts, $X_\Gamma$ is a versal deformation of $Y$, and $B_\Gamma$ is the {\em pure braid group}, consisting of braids which return each strand to its original position. In this case, the braids $t_i^2$ act by family spherical twists: these fibre over the deformation base (Section~\ref{section.geom_descrip_equiv}), deforming the original spherical twists on~$Y$.

The actions given in Theorem~\ref{keythm} also appear in the context of geometric representation theory. They may be obtained from work of Bezrukavnikov--Riche \cite[Section 4]{BezRic}, who construct such actions by representation-theoretic means on slices of the Grothendieck--Springer simultaneous resolution. Similar constructions appear in Cautis--Kamnitzer \cite[Section 2.5]{CauKam}, in the study of categorified quantum group representations. However our focus is somewhat different, as we see these actions as arising from the toric geometry of the varieties $X_\Gamma$, and our main aim is to relate this to the physical heuristics.

We remark very briefly on our method of proof (which is quite different to \cite{BezRic,CauKam} or \cite{ST}). In our construction, the autoequivalences which generate the group action turn out to correspond to `windows' in the derived category of the stack associated to a certain GIT problem. After a careful analysis of these windows, a conceptually simple proof of the braid relations emerges, in Section~\ref{sect.braid_relns_ambient_space}. We hope that this approach will be useful in more general situations, where the machinery of \cite{DHL, BFK} can be applied. See Remark~\ref{rem.philosophy_and_poss_generalizatn} for more discussion on this point.

\subsection{Outline}

The structure of this paper is as follows.

\begin{itemize}
\item Section~\ref{sect.examples_results} explains some simple examples, and then gives a detailed statement of our main result in Section~\ref{section.statement}.
\item Section~\ref{sect.toriccalculations} details the toric geometry of our examples $X_\Gamma$, and of their flops. This yields derived equivalences between phases in Section~\ref{sect.derivedequivalencesfromwalls}.
\item Section~\ref{sect.FIps} recalls the heuristics which allow us to describe the FI parameter spaces for our examples in Section~\ref{sect.MSheuristics}, and explains how to apply them in Section~\ref{sect.resultsForExamples}. We go on to identify the large-radius limits in Section~\ref{sect.LRLs}, and the fundamental groups which we expect to act on the $\Db(X_\Gamma)$ in Section~\ref{sect.fundamentalgroups}.
\item Section~\ref{sect.mixedbraidgroupactions} proves the physical prediction of Section~\ref{sect.FIps}. Specifically, we show that a certain groupoid $T_\Gamma$ acts (faithfully) via derived equivalences and autoequivalences on the phases of $X_\Gamma$. The isotropy group of $T_\Gamma$ is $B_\Gamma$, and so Theorem~\ref{keythm} follows as an immediate corollary. Section \ref{sect.posetGroupoidActions}  proves the expected relationships between the actions for different partitions $\Gamma$.
\end{itemize}

Appendix~\ref{list of notations} lists our main notations, along with cross-references to their explanations.

\begin{acks}This project grew out of enlightening discussions between W.D. and Iain Gordon. W.D. is also grateful to Adrien Brochier, Jim Howie and Michael Wemyss for valuable conversations and helpful suggestions. E.S. would like to thank Tom Coates, Kentaro Hori and especially Hiroshi Iritani for their patient explanations and many extremely useful comments and suggestions.\end{acks}

\section{Examples and results}\label{sect.examples_results}
\subsection{The $A_1$ examples}\label{sect.A1case}
\subsubsection{The $3$-folds}\label{sect.thethreefolds}
To set the stage, we recall some very well-known constructions. Consider the stack
$$\cX =  \quotstack{\C^4_{b_0,a_0,b_1,a_1}}{\C^*} $$
where the $\C^*$ acts with weights $(1,-1,-1,1)$. The coarse moduli space of this stack is the $3$-fold ODP singularity $\set{uv=z_0z_1}$, which is the versal deformation of the $A_1$ surface singularity.\footnote{The variables here are the invariant functions $u=a_0 a_1$, $v=b_0 b_1$, $z_0 = a_0 b_0$, and $z_1 = a_1 b_1$.}  If we pick a character of $\C^*$ then we can form the associated GIT quotient. There are two possible quotients which we'll denote by $X_+$ and $X_-$: each one is isomorphic to the total space of the bundle $\cO(-1)^{\oplus 2}_{\P^1}$. This is the standard Atiyah flop.

These two resolutions $X_+$ and $X_-$ have a common roof
\begin{equation}\label{roofofatiyahflop}\begin{tikzcd}[column sep=8pt,row sep=12pt] \, & \cO(-1,-1)_{\P^1\times \P^1} \arrow[-,transform canvas={xshift=+\sepForEquals}]{d} \arrow[-,transform canvas={xshift=-\sepForEquals}]{d} \\
\, & \tilde{X} \arrow[swap]{ddl}{\pi_{+}} \arrow{ddr}{\pi_{-\phantom{+}}} &\, \\
\\
X_+& \,& X_- \end{tikzcd}\end{equation}

Bondal and Orlov \cite{BO} showed that this correspondence induces a derived equivalence
$$F=(\pi_-)_*(\pi_+)^*: \Db(X_+) \isoto \Db(X_-). $$

There is another approach to this derived equivalence, introduced by the second-named author in \cite{segal} and based on the physics arguments of \cite{HHP}. We view the GIT quotients as open substacks
$$ \iota_{\pm}: X_{\pm}\into \cX, $$
and we define certain subcategories of $\Db(\cX)$, nicknamed `windows', by
$$\cW(k) = \left\langle\,\cO(k), \cO(k+1) \,\right\rangle \subset \Db(\cX)$$
for $k\in \mathbb{Z}$, i.e. $\cW(k)$ is the full triangulated subcategory generated by this pair of line bundles. Then it is easy to show that for any $k$ both functors
$$\iota_\pm^*: \cW(k) \to \Db(X_\pm) $$
are equivalences, so we can define a set of derived equivalences between $X_+$ and $X_-$ by
\begin{equation}\label{eqn.3fold_window_equiv}\psi_k: \Db(X_+) \xrightarrow{(\iota_+^*)^{-1}} \cW(k) \xrightarrow{\phantom{(}\iota_-^*\phantom{)}} \Db(X_-). \end{equation}
The relationship between these two approaches is the following statement:

\begin{prop} \label{prop.windowisflop}The window equivalence $\psi_{-1}$ and the Bondal--Orlov equivalence $F$ coincide:
$$ \psi_{-1} = F $$
\end{prop}
\begin{proof}

Consider the autoequivalence
$$\hat{F}= (\iota_-^*)^{-1} F \iota_+^* : \cW(-1) \to \cW(-1). $$
The proposition is the statement that $\hat{F}$ is the identity. Since $\cW(-1)$ is generated by $\cO(-1)$ and $\cO$, it's enough to check that $\hat{F}$ acts as the identity on these two objects, and on all morphisms between them. For the objects, one can check that $F$ sends $\cO$ to $\cO$ and $\cO(-1)$ to $\cO(1)$, which is the required statement because $\iota_{-}^*\cO(-1) = \cO(1)$. For the morphisms, we note that all morphisms between these two line bundles come from functions on $\C^4$ (there are no higher Exts), and $F$ is the identity away from a subvariety of codimension 2.
\end{proof}
 
The other $\psi_k$ can be obtained by modifying $F$ by the appropriate line bundle on $\tilde{X}$. Also see \cite[Section 3.1]{HLS} for a more general version of this statement.

Now we connect the above with the GLSM heuristics. For this GIT problem, the FI parameter space turns out to be a $\P^1$ with three punctures:
$$\cF_{\cX} = \P^1 - \{1\!:\!0,\; 0\!:\!1,\; 1\!:\!1\}$$
 The first two punctures (or more accurately, neighbourhoods of these punctures) are `large-radius' (LR) limits corresponding to the two phases $X_\pm$. The third puncture is the `conifold point' where the theory becomes singular (we'll explain this picture further in Section~\ref{sect.FIps}). Traversing a loop around one of the large-radius limits is supposed to produce an autoequivalence on the derived category of the corresponding phase, and the correct autoequivalence is just tensoring by the $\cO(1)$ line bundle.

 More interestingly, travelling along a path which starts near one LR limit and ends near the other should produce a derived equivalence between the two phases. Identify $\cF_\cX$ with $\C_\zeta - \set{0,1}$, then take $\log(\zeta)$ to get the infinite-sheeted cover $\C_{\log(\zeta)} - 2\pi i \Z $. The LR limits lie at $\Re(\log(\zeta))\ll 0$ and $\Re(\log(\zeta))\gg 0$, so consider a path that travels from one to the other with $\Re(\log\zeta)$ increasing monotonically. The homotopy class of such a path is determined by which interval it lies in when it crosses the imaginary axis. The interpretation of the physical arguments in \cite{HHP} is that the path that goes through the interval between $2\pi i k$ and $2 \pi i (k+1)$ produces the derived equivalence  $\psi_k$ described above.

Now consider a loop that begins in $\Re(\log(\zeta))\ll 0$ and circles the origin once clockwise, without encircling any other punctures. It follows that this loop produces the autoequivalence 
$$(\psi_{-1})^{-1}\circ \psi_{0}: \Db(X_+) \to \Db(X_+).$$
In \cite{seg-don} autoequivalences of this form were called `window shifts'. It was shown there, for a larger class of examples, that these autoequivalences can be described as twists of certain spherical functors. This particular window shift is equal to a Seidel--Thomas spherical twist \cite{ST} around the spherical object $\mathcal{S} = \cO_{\P^1}(-1)$ in $\Db(X_+)$.\footnote{This particular relation was well-known long before \cite{seg-don}: see \cite[Example 5.10]{Kawamata:2002vq} for a much earlier discussion.} That is, it's equal to a functor
\begin{equation}\label{eqn.ST_twist}T_{\mathcal{S}}: \mathcal{E} \mapsto \cone{ \mathcal{S}\otimes \Hom(\mathcal{S}, \mathcal{E}) \to \mathcal{E} } \end{equation}
Given Proposition~\ref{prop.windowisflop}, we can also describe this autoequivalence as a composition of `flop' and `flop back again' (with the appropriate line bundles inserted).

\subsubsection{The surfaces}\label{sect.thesurfaces}

Now consider the stack
$$\cY =   \quotstack{\C^3_{s_0,s_1,p}}{\C^*}  $$
where the $\C^*$ acts with weights $(1,1,-2)$. The coarse moduli space of this stack is the $A_1$ surface singularity $\set{uv = z^2}$. There are again two possible GIT quotients: $Y_+$ is the total space of $\cO(-2)_{\P^1}$, and $Y_-$ is the orbifold $\quotstack{\C^2}{\Z_2}$. Most of the analysis of the previous example can be repeated verbatim, so there is an equivalence
$$\tilde{F}: \Db(Y_+) \to \Db(Y_-)$$
coming from the common birational roof, and there are windows $\tilde{\cW}(k)\subset \Db(\cY)$ defined in the same way and giving rise to equivalences
$$\tilde{\psi}_k:\Db(Y_+) \to \Db(Y_-). $$
Also, $\tilde{F}=\tilde{\psi}_{-1}$ by the same argument.

 The FI parameter space for this example turns out to be an orbifold: it's a $\P^1$ with two punctures
$$\cF_{\cY} = \P^1  -\{1\!:\!0,\; 1\!:\!1\}$$
and an orbifold point at $0\!:\!1$ with $\Z_2$ isotropy group. The phases $Y_+$ and $Y_-$ correspond to the regions near $1\!:\!0$ and $0\!:\!1$ respectively, and $1\!:\!1$ is again a conifold point. Using the functors $\tilde{\psi}_k$, and tensoring by line bundles, we can again produce an action of $\pi_1(\cF_\cY)$ on the derived category of either of the two phases. The existence of the orbifold point in $\cF_{\cY}$ corresponds to the fact that on $Y_-$, tensoring with $\cO(1)$ has order 2. Also, one can show that looping around the conifold point produces a spherical twist autoequivalence
$$T_{\tilde{\mathcal{S}}}: \Db(Y_+) \to \Db(Y_+), $$
where $\tilde{\mathcal{S}}$ is the spherical object $\cO_{\P^1}(-1)$ in $\Db(Y_+)$.

\subsubsection{Relating the examples}

Now we can ask about the relationship between these two examples. It is a well-known fact that we can include $Y_+$ as a subvariety into either of $X_+$ or $X_-$ as the zero locus of the invariant function $b_0 a_0 - b_1 a_1$ on $\C^4$, so that we have $$j_\pm: Y_+ \into X_\pm.$$ Indeed if we put the obvious symplectic form on $\C^4$ then we can view $Y_+$ as a hyperk\"ahler quotient, and the invariant function as the complex moment map. 

We claim that this fact is reflected in the FI parameter spaces for the two examples. Specifically, we make the observation that there is a 2-to-1 covering map
$$\cF_{\cX} \to \cF_{\cY} $$
sending the conifold point to the conifold point, and identifying the two LR limits in $\cF_{\cX}$ with the limit corresponding to $Y_+$ in $\cF_\cY$. To see this, just notice that we can identify $\cF_\cY$ with
 $$\quotstack{\C_\zeta - \set{0,1}}{\Z_2} $$
where the involution is $\zeta\mapsto 1/\zeta$. We'll explain this more systematically in Section~\ref{sect.FIps}.
 
Now this covering map suggests that the derived equivalences that arise from the two examples are related, in a way that corresponds to the map between the fundamental groupoids of $\cF_\cX$ and $\cF_\cY$. For example, a loop around one of the LR limits in $\cF_\cX$ projects to a loop around the LR limit in $\cF_\cY$, which corresponds to the obvious fact that
$$j_\pm^* \big( (-)\otimes \cO(1)\big)= j_\pm^*(-)\otimes \cO(1). $$
More interestingly, a path between the two LR limits in $\cF_\cX$ projects to a loop around the conifold point in $\cF_\cY$, which suggests the following proposition.
\begin{prop}\label{prop.flopinducestwist} The square below commutes.
$$\begin{tikzcd} \Db(X_+) \rar{\psi_{0}} \dar[swap]{j_+^*} & \Db(X_-) \dar{j_-^*} \\
\Db(Y_+) \rar[swap]{T_{\tilde{\cS}}} & \Db(Y_+) \end{tikzcd} $$
\end{prop}
\begin{proof}
This is very similar to the proof of Proposition~\ref{prop.windowisflop}. It's enough to check the statement on the line bundles $\cO$ and $\cO(1)$, and all morphisms between them. It's easy to check that $T_{\tilde{\cS}}$ sends $\cO$ to $\cO$, and $\cO(1)$ to the cone on the non-trivial extension
$$\cone{\cO_{\P^1}(-1) [-1] \,\to \,\cO(1)} $$
which is $\cO(-1)$, as required. The argument for the morphisms is similar to our previous one. (The $\P^1\subset Y_+$ is codimension 1, but all morphisms between these line bundles do extend uniquely.)
\end{proof}

\begin{rem} One can also prove the above proposition by considering the Fourier--Mukai kernel for $\psi_0$ given to us from its geometric description, calculating the derived restriction of this kernel to $Y_+\times Y_+$, and checking that the result agrees with the kernel for the spherical twist. \end{rem}

We may then deduce the following corollary.
\begin{cor} \label{cor.twistinducestwistsquared} The square below commutes.
$$\begin{tikzcd} \Db(X_+) \rar{T_{\cS}} \dar[swap]{j_+^*} & \Db(X_+) \dar{j_+^*} \\
\Db(Y_+) \rar[swap]{(T_{\tilde{\cS}})^2} & \Db(Y_+) \end{tikzcd} $$
\end{cor}
This fact can also be deduced from the results in \cite{HuyThom}. It corresponds to the fact that a loop around the conifold point in $\cF_\cX$ projects to a \textit{double} loop around the conifold point in $\cF_\cY$.

\begin{rem}\label{rem.SKMSvsFIPS}
It's important not to confuse the FI parameter space with the `stringy K\"ahler moduli space' (SKMS). Under renormalization, a GLSM is believed to flow to a (super)conformal field theory.\footnote{If the associated GIT quotients are non-compact then this statement is problematic, and this is true in the examples that we care about. But we'll skip over this point.} The SKMS is a particular slice through the moduli space of this conformal field theory, obtained by only varying the K\"ahler parameters of the theory. If we assume the Calabi--Yau condition, varying the FI parameter corresponds, after renormalization, to a variation of the K\"ahler parameters, and so there should be a map
$$\cF \longrightarrow \mbox{SKMS}. $$
For the example $\cY$, it is believed that this map is an isomorphism, so $\cF_\cY$ is exactly the K\"ahler moduli space of the associated CFT.

 However, this cannot be true in the $3$-fold example $\cX$. The two phases $X_+$ and $X_-$ are isomorphic and so produce the same CFT, therefore $\cF_\cX$ certainly does not inject into the moduli space of CFTs. For this example, it is believed that the true SKMS is in fact $\cF_\cY$, so the FI parameter space is actually a double cover of the true SKMS.\footnote{In these two examples the map from $\cF$ to the K\"ahler moduli space is at least a local isomorphism, but if we add superpotentials to our GLSMs then it's easy to find examples where this fails.}

  It should be possible to produce derived autoequivalences of the phases $X_\pm$ by performing parallel transport of B-branes over the SKMS, not just over the FI parameter space, so the autoequivalence group that we are seeing is actually an index~2 subgroup of a larger possible group. Concretely, this means that the autoequivalence $T_\cS$ of $\Db(X_+)$ should have a square root. This is indeed true, and the required autoequivalence can be produced by composing $F$ with pullback over the obvious isomorphism between $X_+$ and $X_-$. There is also a beautiful construction of Hori \cite[Section 2.4]{Hori11} that produces it using a different GLSM (equipped with a superpotential). Note however that although $T_\cS$ is compactly supported (i.e. it is trivial away from a compact subvariety), the square root of $T_\cS$ is not.
\end{rem}

\subsection{Statement of results}
\label{section.statement}

In this section we'll explain the class of examples that we're going to consider, and the results that we obtain.

\subsubsection{Construction}
\label{section.construction}

Fix an integer $k\geq 1$. Consider the quiver obtained by taking the affine Dynkin diagram of type $A_k$, and replacing each edge with a pair of arrows, one in each direction (see Figure~\ref{fig.quivers}). We'll label the clockwise arrows by $a_0,\ldots,a_k$, and the anti-clockwise arrows by $b_0,\ldots,b_k$. Now consider the Artin  stack $\quotstack{V}{T}$ parametrizing representations of this quiver which have dimension 1 at each vertex. More explicitly, we let

\begin{itemize}
\item $V = \C^{2(k+1)}$ be the vector space (whose dual is) spanned by the arrows, so that it has co-ordinates $a_0,\ldots,a_k, b_0,\ldots,b_k$, and 
\item $T =(\C^*)^{k+1}$ be the torus with one $\C^*$ factor for each vertex.
\end{itemize}
There's an obvious action of $T$ on $V$, by letting the $\C^*$ associated to the $i^{\text{th}}$ vertex act with
\begin{itemize}
\item weight $+1$ on the two incoming arrows $a_{i-1}, b_i$, and 
\item weight $-1$ on the two outgoing arrows $a_i, b_{i-1}$,
\end{itemize}
reading indices modulo $k+1$. Note that the diagonal $\C^*$ acts trivially, and also that the whole torus acts with trivial determinant on $V$.

%%%%%% BEGIN TikZ
\begin{figure}[h]
                  %%%%%%%%% subfigure
                  \begin{minipage}[b]{.3\linewidth}
                    \centering
\begin{tikzpicture}[>=stealth,scale=2,node distance=6]
\quiver{->}{->}{->}{->}{->}{->}{0}{0}{basic} ;
 \node[above of=basicV1] {\scriptsize$1$} ;
 \node[above of=basicV2] {\scriptsize$2$} ;
 \node[below of=basicV0] {\scriptsize$0$} ;
\end{tikzpicture}
\vspace{0.3cm}
\subcaption{$A_2$ quiver}\label{fig.quiver2}
                  \end{minipage}%
                   %%%%%%%%% subfigure
                 \begin{minipage}[b]{.7\linewidth}
                    \centering
\begin{tikzpicture}[>=stealth,scale=2.5,node distance=7]
% vertices
\foreach \i in {1,...,5} {
	\node(V\i) at (\i,0) [vertex] {};
	\node[above of=V\i] {\scriptsize$\i$};
	} ;
% frame vertex
\node(V0) at (3,-1) [vertex] {};
\node[below of=V0] {\scriptsize$0$};
% arrows
\foreach \i/\j/\bshift/\ashift/\bpos/\apos/\bend in {
	5/4/.5/-.5///10,
	4/3/.5/-.5///13,
	3/2/.5/-.5///13,
	2/1/.5/-.5///10,
	1/0/-.5/.5///7,
	0/5/-.5/.5///7} {
	\draw[->,bend left=\bend,looseness=1] (V\i) to node[gap, \bpos ,pos=0.5]{\scriptsize$b_\j$} (V\j);
	\draw[->,bend left=\bend,looseness=1] (V\j) to node[gap, \apos ,pos=0.5]{\scriptsize$a_\j$} (V\i);
	} ;

\end{tikzpicture}
\subcaption{$A_5$ quiver}\label{fig.quiver5}
                  \end{minipage}
                  \caption{Quivers associated to affine Dynkin diagrams.} \label{fig.quivers}
                \end{figure}
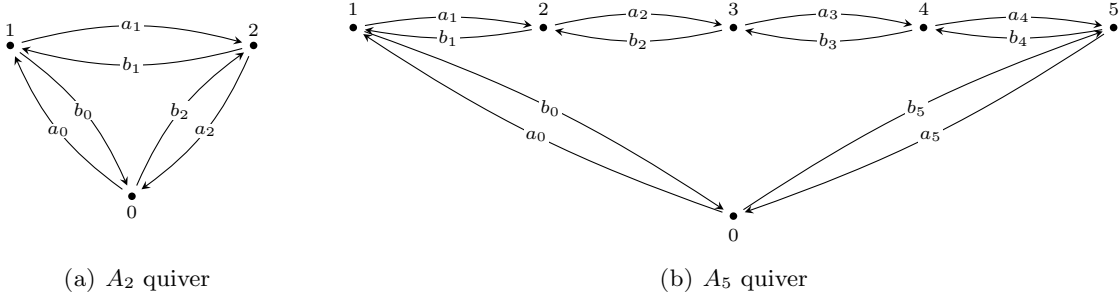
%%%%%% END TikZ

The $T$-invariant functions on $V$ are generated by the monomials
$$ u = a_0a_1\ldots a_k, \hspace{1cm} v = b_0b_1\ldots b_k $$
and 
$$ z_i = a_ib_i, \;\; \mbox{for } i\in [0,k]. $$
So the affine quotient $V/T$ is the singularity
$$uv = z_0z_1\ldots z_k $$
which is the universal unfolding of the $A_k$ surface singularity.

If we choose a character $\theta$ of $T$, we can take a GIT quotient 
$$X = V \sslash_\theta T.$$
For $X$ to be non-empty, we need to choose a $\theta$ that annihilates the diagonal $\C^*$. Then, for a generic such $\theta$, the quotient $X$ is a smooth toric variety resolving the singularity $V/T$. It's also Calabi--Yau, since $T$ acts through $\operatorname{SL}(V)$, and as we shall see later it's independent of $\theta$, i.e. all the phases are isomorphic.

\pgap

This construction is well-known in the context of Nakajima quiver varieties. In that approach, one equips $V$ with a $T$-invariant symplectic form by making $a_i$ and $b_i$ symplectically dual, then takes a hyperk\"ahler quotient $Y$. As a complex variety, $Y$ is a subvariety of $X$ formed by taking level sets of all the `complex moment maps', which are the invariant functions
\begin{equation}\label{eqn.moment_maps}\mu_i = z_i - z_{i-1} \end{equation}
for $1\leq i \leq k $. Using these functions, we can view $X$ as a family
$$ X \stackrel{\mu}{\longrightarrow} \C^k. $$
The fibres are all smooth surfaces: these are the (underlying varieties of the) possible hyperk\"ahler quotients $Y$. The fibre over zero $Y_0$ is the very famous  small resolution of the $A_k$ surface singularity
$$uv = z^{k+1},$$
which appears in the McKay correspondence and Kronheimer's ALE classification. The larger space $X$ is a versal deformation of $Y_0$. Note that it may also be obtained as the inverse image of a Slodowy slice under the Grothendieck--Springer resolution \cite{Slodowy}.

\pgap

 In the case $k=1$, we get the $3$-fold $X$ and the surface $(Y_0=) Y_+ \subset X$  that we discussed in Section~\ref{sect.A1case}.  For higher $k$, we are also interested in some intermediate subvarieties lying between $Y_0$ and $X$. These intermediate subvarieties are indexed by partitions, as we will now describe.

Firstly, note that for any $i,j\in[0,k]$ the invariant function $z_i - z_j$ lies in the space $\C^k$ spanned by the complex moment maps. Now let $\Gamma$ be a partition of the set
$$\set{z_0,\ldots,z_k}.$$
There is a corresponding subspace of  $\C^k$, where we include the function $z_i-z_j$ if and only if the variables $z_i$ and $z_j$ lie in the same part of $\Gamma$. We let
$$X_\Gamma\subset X$$
be the subvariety defined as the vanishing locus of the subspace of the complex moment maps corresponding to a partition $\Gamma$.

Partitions form a poset, ordered by refinement, and obviously we have 
$$X_\Gamma \subset X_{\Gamma'} $$
if $\Gamma'$ is a refinement of $\Gamma$. We'll let $\Gamma_{\fin}$ denote the finest possible partition, which has $(k+1)$ parts, then $X_{\Gamma_{\fin}}$ is the ambient space $X$. At the other extreme, the coarsest possible partition $\Gamma_{\crs}$, which has only one part, corresponds to the surface $X_{\Gamma_{\crs}}= Y_0$. 

\subsubsection{Physical heuristics}
\label{sect.heuristics}
 
 As we will justify in Section~\ref{sect.toriccalculations}, each of these varieties $X_\Gamma$ is a smooth toric Calabi--Yau, which arises as the GIT quotient of a vector space by a torus. As such, we can view each one as a phase of an abelian GLSM, and so we can run the physicists' recipe and compute the FI parameter space $\cF_\Gamma$ for each one. We will do this in Section~\ref{sect.FIps}, and the results are as follows.

 For the ambient space $X=X_{\Gamma_{\fin}}$, the FI parameter space is the set
$$\cF_{\fin} = \setconds{(\zeta_0\!:\!\ldots\!:\!\zeta_k)}{\begin{aligned} &\zeta_i\neq 0 \; \forall i \\[-5pt] &\zeta_i\neq \zeta_j\; \forall i\neq j \end{aligned}} \;\subset\; \P^k $$
of $(k+1)$-tuples of distinct non-zero complex numbers, up to overall scale. Now take a partition $\Gamma$, and let 
$$S_\Gamma \subset S_{k+1}$$
be the Young subgroup of permutations that preserve $\Gamma$. As we shall show, the FI parameter space associated to the subvariety $X_\Gamma \subset X$ is the orbifold
$$\cF_\Gamma = \quotstack{\cF_{\fin}}{S_\Gamma} $$
using the obvious action of $S_{k+1}$ on $\cF_{\fin}$. So the FI parameter spaces also show the poset structure, since we have a covering map
$$\cF_{\Gamma'} \to \cF_\Gamma$$
whenever $\Gamma'$ is a refinement of $\Gamma$.

  The FI parameter space associated to the surface $Y_0$ is $\cF_{\crs} = \quotstack{\cF_{\fin}}{S_{k+1}}$, which is a $\C^*$ quotient of the configuration space of $k+1$ points in $\C^*$. The fundamental group of this space is generated by two subgroups: a lattice, generated by letting each $\zeta_i$ loop around zero, and a copy of the braid group $B_{k+1}$. For our purposes the lattice is not very interesting, and we'll focus on the braid group. The heuristic picture of brane transport over the FI parameter space predicts that there should be an action
$$B_{k+1} \curvearrowright \Db(Y_0).$$
This braid group action was famously constructed by Seidel--Thomas  \cite{ST}. 

Now choose a general partition $\Gamma$. It's clear that (the interesting part of) the fundamental group of $\cF_\Gamma$ is the group $B_\Gamma$, defined as the fibre product
\begin{equation}\label{eqn.mixed_braid_defn}\begin{tikzcd}B_\Gamma \rar[hook] \dar & B_{k+1} \dar \\
S_\Gamma \rar[hook] & S_{k+1} \end{tikzcd}\end{equation}
where $S_\Gamma$ is the Young subgroup of elements of $S_{k+1}$ preserving the partition $\Gamma$. This group $B_\Gamma$ is sometimes called a `mixed braid group': it consists of braids that permute their endpoints in a way that preserves the partition $\Gamma$. The special case $\Gamma=\Gamma_{\fin}$, and so $S_\Gamma=\{1\}$, produces the `pure braid group' $P_{k+1}$. 

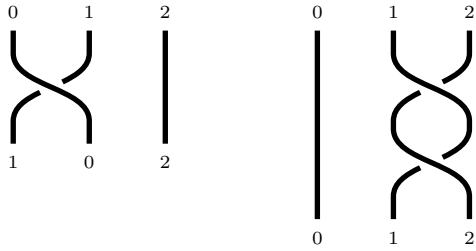
\begin{figure}[H]
\begin{tikzpicture}
	\braid[number of strands=3,line width=2pt] (B1) at (-2,0) s_1;
	\node[at=(B1-1-s), yshift=7pt] {\scriptsize 0};
          \node[at=(B1-2-s),  yshift=7pt] {\scriptsize 1};
	\node[at=(B1-3-s),  yshift=7pt] {\scriptsize 2};
	\node[at=(B1-1-e), yshift=-7pt] {\scriptsize 0};
          \node[at=(B1-2-e),  yshift=-7pt] {\scriptsize 1};
	\node[at=(B1-3-e),  yshift=-7pt] {\scriptsize 2};

	\braid[number of strands=3,line width=2pt] (B2) at (+2,0) s_2 s_2;
	\node[at=(B2-1-s), yshift=7pt] {\scriptsize 0};
          \node[at=(B2-2-s),  yshift=7pt] {\scriptsize 1};
	\node[at=(B2-3-s),  yshift=7pt] {\scriptsize 2};
	\node[at=(B2-1-e), yshift=-7pt] {\scriptsize 0};
          \node[at=(B2-2-e),  yshift=-7pt] {\scriptsize 1};
	\node[at=(B2-3-e),  yshift=-7pt] {\scriptsize 2};
\end{tikzpicture}
\caption{Elements of the mixed braid group $B_\Gamma$ for the partition $\Gamma = (0 1)(2)$.}
\label{fig.mixedbraid}
\end{figure}

The FI parameter space heuristics suggest the following result, which we shall prove:
\begin{thm}[Corollary~\ref{cor.maincor}]\label{thm.mainthm_in_intro_section} There is an action of the mixed braid group $B_\Gamma$ on the derived category $\Db(X_\Gamma)$.\end{thm}
In the surface case, Seidel--Thomas proved the rather deep result that the action is faithful. We can leverage their result to prove that all our actions are also faithful.

 Note that in the $k=1$ case  (Section~\ref{sect.A1case}), we did indeed see the above structure, but there wasn't much to prove since both $B_2$ and $P_2$ are isomorphic to $\Z$. The action of $\Z$ on both $\Db(X_+)$ and $\Db(Y_+)$ was generated in each case by a spherical twist. The non-trivial fact was how these two actions were related to each other: we saw (Corollary~\ref{cor.twistinducestwistsquared}) that the action of $P_2$ on $\Db(X_+)$ intertwined, via restriction, with the action of the subgroup $P_2\subset B_2$ on $\Db(Y_+)$. Heuristically, this was a reflection of the fact that the FI parameter space for $X_+$ was a double cover of the FI parameter space for $Y_+$.

 For general $k$, similar FI parameter space heuristics predict that our mixed braid group actions should be related to each other, by the same poset structure. Precisely, we should expect that  if $\Gamma'$ is a refinement of $\Gamma$, then the action of $B_{\Gamma'}$ on $\Db(X_{\Gamma'})$ will intertwine via the restriction functor
$$\Db(X_{\Gamma'}) \to \Db(X_\Gamma)$$
 with the action of the subgroup $B_{\Gamma'}\subset B_{\Gamma}$ on $\Db(X_\Gamma)$. In the course of our proof we will show that this prediction does indeed hold (Proposition~\ref{prop.groupoid_action_intertwinement}).

\begin{rem} As we explained in Remark~\ref{rem.SKMSvsFIPS}, the FI parameter spaces are not the true K\"ahler moduli spaces for these theories, indeed the true SKMS appears to be always given by $\cF_{\crs}$ (for a fixed value of $k$). On the level of derived categories, this means that our mixed braid group actions could be extended to an action of the full braid group, by including square roots of the relevant spherical twists.
\end{rem}

\subsection{Future directions}

In principle, one can carry out this program for any example of a torus acting on a vector space. The main obstacle appears to be finding a meaningful description of the group $\pi_1(\cF)$. 

A first guess for a good generalization of our examples is to follow the standard technique for Nakajima quiver varieties, and replace the affine $A_k$ Dynkin diagram by some other graph. Unfortunately, as soon as the graph has vertices with valency greater than $2$, the obvious subvarieties of the representation space (the analogues of our $X_\Gamma$'s) will not be toric, and so we won't get the whole poset of GLSMs and covering maps that we see in our examples.

 In fact this covering map phenomenon, and its relationship with derived categories, is probably the most interesting feature of our construction. It would be worthwhile to investigate the general conditions under which it arises.

\section{Toric calculations}\label{sect.toriccalculations}

In this section we'll apply some completely standard toric techniques to understand our varieties~$X_\Gamma$.

\subsection{Representations of the free quiver}\label{sect.reps_free_quiver}

Our first task is to find the toric fan for the ambient variety $X$. As an initial step, we can package our construction as an exact sequence of lattices as follows:

%%%%%% BEGIN TikZ
\begin{center}
\begin{tikzcd}[column sep=25pt, row sep=15pt] \Z \ar{r}{\phantom{1}\left(\begin{smallmatrix}1\pgfmatrixnextcell\cdots\pgfmatrixnextcell1\end{smallmatrix}\right)\phantom{1}} & \Z^{k+1} \ar{r}{Q} & \Z^{2(k+1)} \ar{r}{P} & \Z^{2+(k+1)} \ar{r}{\left(\begin{smallmatrix}+1\\+1\\-1\\[-3pt]
\scriptsize{\tvdots}\\[-3pt]
-1 \end{smallmatrix}\right)} & \Z \\ & \Hom(\C^*,T) \arrow[-,transform canvas={xshift=+\sepForEquals}]{u} \arrow[-,transform canvas={xshift=-\sepForEquals}]{u} & \text{\small basis vectors $b_i$, $a_i$} \ar[squiggly]{u} & \text{\small co-ordinates $u$, $v$, $z_i$} \ar[squiggly]{u} \end{tikzcd}
\end{center}

\begin{figure}[h]
\begin{center}
\newlength{\mycolwdC}% array column width
\settowidth{\mycolwdC}{$\scriptstyle{a_{k-1}}$}% width of widest element in array
\BAnewcolumntype{C}{>{\centering$}p{\mycolwdC}<{$}}
$$Q\quad=\quad\begin{blockarray}{CCCCCCCCCCCcc}\scriptstyle{b_0}&\scriptstyle{a_0}&\scriptstyle{b_1}&\scriptstyle{a_1}&\scriptstyle{b_2}&\scriptstyle{a_2}&\scriptstyle{\cdots}&\scriptstyle{b_{k-1}}&\scriptstyle{a_{k-1}}&\scriptstyle{b_k}&\scriptstyle{a_k}& \\[1ex]
 \begin{block}{(CCCCCCCCCCC)cc} -1&+1&+1&-1&0&0&\cdots&0&0&0&0 &&\scriptstyle{1} \\ 0&0&-1&+1&+1& -1&  &0&0&0&0 &&\scriptstyle{2}\\ \vdots&&&&&&&&&&\vdots&&\scriptstyle{\vdots}\\0&0&0&0&0&0&  &-1&+1&+1&-1 &&\scriptstyle{k}\\+1&-1&0&0&0&0&\cdots&0&0&-1&+1 &&\scriptstyle{0}\\ \end{block} \end{blockarray}$$
$$
\newlength{\mycolwdD}% array column width
\settowidth{\mycolwdD}{$\scriptstyle{z_2}$}% width of widest element in array
\BAnewcolumntype{D}{>{\centering$}p{\mycolwdD}<{$}}
\def\hspacing{10pt}
P\quad=\quad\begin{blockarray}{D@{\hspace{\hspacing}}D@{\hspace{\hspacing}}D@{\hspace{\hspacing}}D@{\hspace{\hspacing}}D@{\hspace{\hspacing}}c@{\hspace{\hspacing}}ccc}\scriptstyle{u}&\scriptstyle{v}&\scriptstyle{z_0}&\scriptstyle{z_1}&\scriptstyle{z_2}&\scriptstyle{\cdots}&&\\[1ex]
\begin{block}{(D@{\hspace{\hspacing}}D@{\hspace{\hspacing}}D@{\hspace{\hspacing}}D@{\hspace{\hspacing}}D@{\hspace{\hspacing}}c@{\hspace{\hspacing}}c)cc}
0&1&1&0&0&\cdots&&&\scriptstyle{b_0}\\
1&0&1&0&0&&&&\scriptstyle{a_0}\\
0&1&0&1&0&&&&\scriptstyle{b_1}\\
1&0&0&1&0&&&&\scriptstyle{a_1}\\
0&1&0&0&1&&&&\scriptstyle{b_2}\\
1&0&0&0&1&&&&\scriptstyle{a_2}\\
\vdots&&&&&&&&\scriptstyle{\vdots}\\
&\\\end{block}\end{blockarray}$$
\end{center}
\caption{Toric data for GIT quotient $X$.}\label{fig.toric}
\end{figure}
%%%%%% END TikZ

The matrices for the maps $Q$ and $P$ are given in Figure~\ref{fig.toric}. The $\Z$ at the left of the sequence corresponds to the diagonal $\C^*\subset T$, which acts trivially on $V$. The toric fan for $X$ lies in the rank $k+2$ lattice $\Im P\subset \Z^{2+(k+1)}$. However, we will often view it as a fan in $ \Z^{2+(k+1)}$, since the larger lattice has a convenient system of co-ordinates.

 For any choice of $\theta$ we know the rays in the toric fan immediately: they're generated by the images of the standard generators for $\Z^{2(k+1)}$ (i.e. the rows of $P$). Consider the generators corresponding to the arrows $a_0,\ldots,a_k$, and look at their images under $P$. This gives a set of vectors $\polyTopVec{0},\ldots,\polyTopVec{k}$ (i.e. the $a$-indexed rows of $P$) which form the vertices of a standard $k$-simplex in the affine subspace $u=1, v=0$. The other half of the generators give a set of vectors $\polyBaseVec{0},\ldots,\polyBaseVec{k}$ (i.e. the $b$-indexed rows of $P$) which span a standard $k$-simplex in $u=0, v=1$. So all the generators together span a polytope
 $$\Pi \cong \Delta^k\times I.$$
 
 \begin{rem}Notice that $\Pi$ lies in an affine hyperplane of height 1, namely $\{z_0+\ldots+z_k=1\}$: this is equivalent to the Calabi--Yau condition.\end{rem}

Now choose a character $\theta=(\theta_0, \ldots,\theta_k) \in (\Z^{k+1})^\vee$ 	of the torus $T$ which annihilates the diagonal $\C^*$, and lift it to an element $\vartheta\in (\Z^{2(k+1)})^\vee$. We can choose $\vartheta$ to be of the form
\begin{equation}\label{thetaform}\vartheta = (0,\vartheta_0 , 0, \vartheta_1, 0, \vartheta_2, \ldots, 0,\vartheta_k).\end{equation}
Then 
$$\theta_0=\vartheta_k - \vartheta_0, \;\;\;\; \theta_1 = \vartheta_0 - \vartheta_1, \;\;\;\; \ldots,\;\;\;\; \theta_k = \vartheta_{k-1} - \vartheta_k, $$
and so these values $\vartheta_i$ are unique up to adding an overall constant. Choose a character such that
\begin{equation}\label{standardtheta}\vartheta_0>\vartheta_1>\ldots>\vartheta_k \end{equation}
i.e. all the $\theta_i$ are positive for $i\geq 1$. We'll refer to the corresponding quotient $X$ as the \textit{standard phase}. 

The toric fan for $X$ is the set of cones on some subdivision of the polytope $\Pi$. To find out which subdivision it is, we use a standard shortcut from toric geometry \cite[Section 3.4]{fult}. View $\vartheta$ as an integer-valued function on the vertices of $\Pi$. After subdividing $\Pi$ according to the toric fan for $X$, this function $\vartheta$ will extend to a strictly concave piecewise linear (PL) function over $\Pi$.
 
This shortcut lets us guess the answer. There is a standard way to triangulate $\Pi =\Delta^k\times I$ into $(k+1)$-simplices, by cutting it into the pieces
\begin{equation}\label{Deltai} \Delta_i = \simplex{\polyTopVec{0}, \ldots, \polyTopVec{i}, \polyBaseVec{i}, \ldots,\polyBaseVec{k}}, \end{equation}
where $0\leq i\leq k $. Our standard $\vartheta$ gives a strictly concave PL function on this subdivision, and this is the only subdivision for which this is true, therefore this subdivision gives the correct fan for the standard phase. In summary:
\begin{lem}
The toric fan for the standard phase  $X$ is given by all the cones on the simplices $\Delta_i$.
\end{lem}
\begin{rem}\label{rem.cone} For future reference in Section~\ref{sect.subvarieties}, note that the cone on the polytope $\Pi$ (i.e. the union of all the cones in the fan) is just the intersection of the positive orthant 
$$\set{u,v,z_0,\ldots,z_k\geq 0}\subset \Z^{2+(k+1)}$$
 with the hyperplane
\begin{equation}\label{hyperplaneImP} u + v = z_0 + \ldots + z_k  \end{equation} 
spanned by the lattice $\Im P$. To get the cone on the simplex $\Delta_i \subset \Pi$, we additionally impose the inequalities
\begin{align}\label{inequalitiesforDeltai}
u\geq z_0+\ldots+z_{i-1}, \hspace{1cm} v\geq  z_{i+1} +\ldots+z_k.
\end{align}
\end{rem}

If we vary the character $\vartheta$, then the above argument shows that the GIT quotient does not change as long as we remain in the region \eqref{standardtheta}, and conversely as soon as two of the $\vartheta_i$ become equal then we hit a wall. Consequently the chambers for the GIT problem are the regions
$$\vartheta_{\sigma(0)}>\vartheta_{\sigma(1)}>\ldots > \vartheta_{\sigma(k)} $$
for some permutation $\sigma\in S_{k+1}$. We'll denote the corresponding quotients by $X^{\sigma}$, but we'll continue to let $X=X^{(1)}$ denote the standard phase.

To get the fan for a non-standard phase $X^{\sigma}$, we just re-order the vertices of $\Pi$ by $\sigma$ before performing the standard triangulation. Notice that all the phases are in fact isomorphic, because of this $S_{k+1}$ symmetry (i.e. the Weyl group). This is not immediately obvious from the original quiver description.

\begin{eg}\label{eg.3foldfan}
{(1)}
 Let $k=1$. The  polytope $\Pi$ lies in the affine subspace
$$\{u + v= z_0 + z_1, \;\;\;\;z_0+z_1 = 1\}\subset \C^4.$$
We can use $(z_1, u)$ as co-ordinates on this subspace, then we can draw the triangulation of $\Pi$ (see Figure~\ref{fig:stdTriangulations}\ref{3foldfan}).
\end{eg}

\noindent {(2)}
 For $k=2$, the  polytope $\Pi$ lies in the affine subspace
$$\{u + v= z_0 + z_1 + z_2, \;\;\;\;z_0+z_1+z_2 = 1\}\subset \C^5.$$
See Figure~\ref{fig:stdTriangulations}\ref{4foldfan} for the triangulation of $\Pi$.

%%%%%% BEGIN TikZ
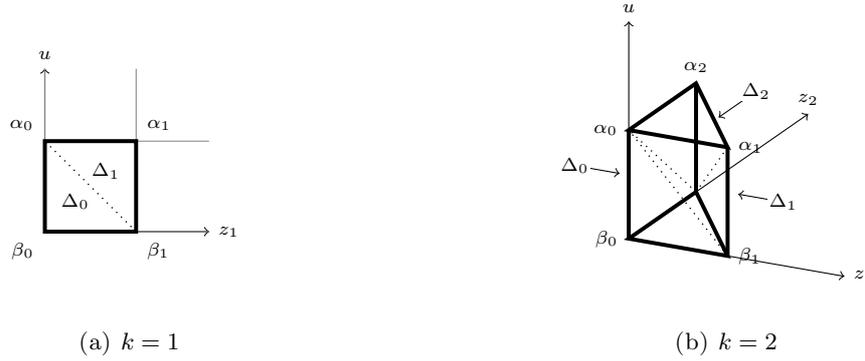
\begin{figure}[h]
                  %%%%%%%%% subfigure
                  \begin{minipage}[b]{.5\linewidth}
                    \centering
\def\axes{1.8}
\begin{tikzpicture}[scale=1.2]
\draw[->] (0,0) -- (\axes,0) node[right]{\scriptsize $z_1$};
\draw[->] (0,0) -- (0,\axes) node[above]{\scriptsize $u$};
\draw[help lines] (0,0) grid (\axes,\axes);

% polytope
\draw[line width=1.5pt]
	(0,0) node[below left]{\scriptsize $\polyBaseVec{0}$} --
	(1,0) node[below right]{\scriptsize $\polyBaseVec{1}$} --
	(1,1) node[above right] {\scriptsize $\polyTopVec{1}$} --
	(0,1) node[above left] {\scriptsize $\polyTopVec{0}$} -- cycle;

% triangulation
\def\pos{1/3}
\draw[triangulation] (1,0) -- (0,1);
\node at (\pos,\pos) {\scriptsize $\Delta_0$};
\node at (1-\pos,1-\pos) {\scriptsize $\Delta_1$};

\end{tikzpicture}
\vspace{0.6cm}
                    \subcaption{$k=1$}\label{3foldfan}
                  \end{minipage}%
                   %%%%%%%%% subfigure
                 \begin{minipage}[b]{.5\linewidth}
                    \centering
                    \def\axes{2.4}

\begin{tikzpicture}[scale=1.2]
\draw[->] (0,0) -- (\angleA:\axes) node[right]{\scriptsize $z_1$};
\draw[->] (0,0) -- (\angleB:\axes) node[above]{\scriptsize $z_2$};
\draw[->] (0,0) -- (0,\axes) node[above]{\scriptsize $u$};
% \draw[help lines] (0,0) grid (\axes,\axes);

% polytope
\polytopeStdTriangulation{\angleA}{\scaleA}{\angleB}{\scaleB}{0}{0}{yes}{basic}{\scaleC};
	
\end{tikzpicture}
\vspace{0.3cm}
                    \subcaption{$k=2$}\label{4foldfan}
                  \end{minipage}
                  \caption{Triangulations of the polytope $\Pi$ whose cones give the fan for the standard phase $X$.} \label{fig:stdTriangulations}
                \end{figure}
%%%%%% END TikZ

% \end{eg}
\subsection{Subvarieties associated to partitions}\label{sect.subvarieties}

In this section, we'll investigate the subvarieties $X_\Gamma\subset X$ cut out by subsets of the complex moment map equations, as defined in Section~\ref{section.construction}.

To start with, suppose $X_\Gamma$ is the divisor $\{\mu_i = 0\}$. The function $\mu_i$ is not a character of the torus which acts on $X$, so $X_\Gamma$ is not a torus-invariant divisor. Nevertheless, $X_\Gamma$ \textit{is} a toric variety. To see this, we first restrict our attention to the Zariski open torus $(\C^*)^{k+2} \subset X $. Within this locus, the set $\{\mu_i=0\}$ is just a corank 1 subtorus. So $X_\Gamma$ is the closure of this subtorus, and is invariant under the action of the subtorus on $X$, so $X_\Gamma$ is a toric variety. Furthermore, to get the toric fan for $X_\Gamma$ we take the toric fan for $X$ and slice it along the corresponding corank 1 sublattice, namely $\{z_i=z_{i-1}\} \subset \Z^{2+(k+1)}$.

 Similarly, if $X_\Gamma\subset X$ is a subvariety of higher codimension (associated to a partition with fewer parts), then $X_\Gamma$ is also a toric variety, and we can obtain its toric fan by slicing along the appropriate sublattice. We'll now perform this procedure explicitly, and work out the toric fans for all the $X_\Gamma$.

\subsubsection{Toric data}

Fix a partition $\Gamma$, and encode it as a surjective function
$$\gamma: [0,k] \to [0,s] $$
whose level sets are the pieces of the partition. The fan for $X_\Gamma$ lies in the sublattice
\begin{equation}\label{eq.sublattice}\Z^{2+(s+1)} \subset \Z^{2+(k+1)} \end{equation}
cut out by the equations $z_i = z_j$, for all $i$, $j$ such that $\gamma(i)=\gamma(j)$.

\begin{notn}\label{notn.coords} We choose co-ordinates $$(\tilde{u},\tilde{v},\tilde{z}_0,\ldots,\tilde{z}_s)$$ on the sublattice $\Z^{2+(s+1)}$ of \eqref{eq.sublattice}, induced from co-ordinates on the bigger lattice $\Z^{2+(k+1)}$ in the natural way: we set $\tilde{u}=u$, $\tilde{v}=v$, and for each $t\in [0,s]$ we have a co-ordinate $\tilde{z}_t$, which is the restriction of $z_{i}$ for all $i\in \gamma^{-1}(t)$. We write
 $$N_t:= \#(\gamma^{-1}(t)) $$
for the sizes of the pieces of the partition.
\end{notn}
\begin{eg}\label{eg.subvariety_coords}  Let $k=4$, and let $\Gamma=(0 1)(2 3 4)$. Then $s=1$, and the corresponding function $\gamma$ and the associated co-ordinates $\tilde{z}_i$ are as follows:
%%%%%% BEGIN TikZ
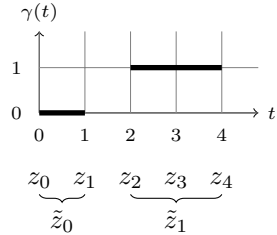
\pgap
\begin{center}
\begin{tikzpicture}[scale=0.6]
% grid:
\draw[->] (0,0) -- (4.8,0) node[right]{\scriptsize $t$};
\draw[->] (0,0) -- (0,1.8) node[above]{\scriptsize $\gamma(t)$};
\draw[help lines] (0,0) grid (4.8,1.8);
\foreach \x in {0,...,4} {
	\draw node at (\x,-0.5) {\scriptsize $\x$};
	} ;
\foreach \y in {0,...,1} {
	\draw node at (-0.5,\y) {\scriptsize $\y$};
	} ;
	 
% function (gamma):
\draw [line width=2pt] (0,0) -- (1,0);
\draw [line width=2pt] (2,1) -- (4,1);

% co-ords
\def\pos{-1.5}
\foreach \index in {0,...,4} {
	\draw node (coordz\index) at (\index,\pos) {$z_\index$};
	} ;
	
% new co-ords
\def\posNew{\pos-0.3}
\draw [decorate,decoration={brace,mirror,amplitude=3pt}] 
(0,\posNew) -- (1,\posNew) node (coordztilde) [black,midway,yshift=-10pt] {$\tilde{z}_0$} ; 
\draw [decorate,decoration={brace,mirror,amplitude=3pt}] 
(2,\posNew) -- (4,\posNew) node [black,midway,yshift=-10pt] {$\tilde{z}_1$} ; 

\end{tikzpicture}
\end{center}
%%%%%% END TikZ
\end{eg}
Recall from Remark~\ref{rem.cone} that the union of all the cones in the fan for $X$ is the intersection of the positive orthant in $\Z^{2+(k+1)}$ with the hyperplane \eqref{hyperplaneImP}. Consequently, the union of all the cones in the fan for $X_\Gamma$ must be the intersection of the positive orthant 
$$\set{\tilde{u}, \tilde{v}, \tilde{z}_0,\ldots,\tilde{z}_s \geq 0 } \subset \Z^{2+(s+1)} $$
with the hyperplane
\begin{equation}\label{hyperplaneforGamma} \tilde{u} + \tilde{v} = \sum_{t=0}^s N_t \tilde{z}_t. \end{equation}
 
The lattice points that lie in this locus are easy to determine, they form a cone generated by the set of points
\begin{equation}\label{eq.polyVecDef}\polyVec{t}{\delta} := (\delta, N_t-\delta, 0, \ldots, 0, \underset{\stackrel{\uparrow}{\scriptstyle t}}{1}, 0, \ldots, 0)\end{equation}where $t\in [0,s]$, and $\delta\in [0, N_t]$. The convex hull of all these generating points $\polyVec{t}{\delta}$ is a  polytope
\begin{equation}\label{eq.prism2}\tilde{\Pi} \cong \Delta^s\times I\end{equation}
which again is isomorphic to a prism on a simplex. The points at the vertices of $\tilde{\Pi}$ are the two $s$-simplices
$$\simplex{\polyVec{0}{0},\polyVec{1}{0},\ldots,\polyVec{s}{0}}
\hspace{1cm}\mbox{and}\hspace{1cm} 
\simplex{\polyVec{0}{N_1},\polyVec{1}{N_2},\ldots,\polyVec{s}{N_s}}$$
which are embedded in a non-standard way. All the other points $\polyVec{t}{\delta}$ lie along edges of $\tilde{\Pi}$.

 We've determined that the union of all the cones in the fan for $X_\Gamma$ is the cone on this polytope $\tilde{\Pi}$. To determine the individual cones, we recall that the polytope $\Pi$ for $X$ is subdivided into $(k+1)$-simplices $\Delta_0,\ldots,\Delta_k$, and the cones on these simplices define the toric fan for $X$.  To get the cones in the fan for $X_\Gamma$, we take the cone on each simplex $\Delta_i$ and slice it along the hyperplane \eqref{hyperplaneforGamma}. 

The cone on $\Delta_i$ is cut out of the positive orthant in $\Z^{2+(k+1)}$ by the inequalities  \eqref{inequalitiesforDeltai}, so after slicing we obtain the subset of the positive orthant in $\Z^{2+(s+1)}$ cut out by the inequalities
\begin{align*}
\tilde{u} \geq \sum_{t=0}^{s} n^i_t \tilde{z}_t, \hspace{1cm} \tilde{v} \geq \sum_{t=0}^{s} m^i_t \tilde{z}_t,
 \end{align*}
where the integers $n^i_t$ and $m^i_t$ are defined as follows:
\begin{eqnarray*}n^i_t = \#\left(\gamma^{-1}(t)\cap [0,i-1] \right), 
\hspace{1cm} m^i_t &=&  \#\left(\gamma^{-1}(t)\cap [i+1,k] \right) \\
&=& \left\{ \begin{array}{cc} N_t-n^i_t, &\mbox{for}\,\, t\neq \gamma(i) \\
 N_{\gamma(i)} -n^i_{\gamma(i)} -1, & \mbox{for}\,\, t =\gamma(i) \end{array}\right. 
\end{eqnarray*}
The lattice points that lie in this region necessarily form a cone generated by some subset of the lattice points in $\tilde{\Pi}$. This subset consists of
\begin{equation}\label{tildeDeltai}\text{$\polyVec{0}{n^i_0}$, $\polyVec{1}{n^i_1}$, $\ldots$ , $\polyVec{s}{n^i_s}$, and $\polyVec{\gamma(i)}{n^i_{\gamma(i)}+1}$.}\end{equation}
These points span an  $(s+1)$-simplex, which we'll denote by $\tilde{\Delta}_i$. Together these simplices form a triangulation of the polytope $\tilde{\Pi}$. In summary:

\begin{lem}\label{lem.raygenerators} The toric fan for $X_\Gamma$ consists of the cones on the simplices $\tilde{\Delta}_i$. In particular, the complete set of points $\polyVec{t}{\delta}$, as defined by \eqref{eq.polyVecDef}, generate all the rays.\end{lem}

\begin{rem}The whole of $\tilde{\Pi}$ lies in the affine hypersurface
$$\set{\tilde{z}_0 + \ldots + \tilde{z}_s = 1}, $$
and hence $X_\Gamma$ is Calabi--Yau.\end{rem}

\begin{rem}\label{rem.visualize}This triangulation of $\tilde{\Pi}\cong \Delta^s\times I$ is easy to visualize: start with the base $s$-simplex 
$$\simplex{\polyVec{0}{0},\ldots,\polyVec{s}{0}},$$
and turn it into an $(s+1)$-simplex by adjoining the point $\polyVec{\gamma(0)}{1}$. Now take the `top' face of this $(s+1)$-simplex, namely
$$\simplex{\polyVec{0}{0},\ldots,\polyVec{\gamma(0)}{1},\ldots,\polyVec{s}{0}},$$
and extend it to a new $(s+1)$-simplex by adjoining the point $\polyVec{\gamma(1)}{n^2_{\gamma(1)}}$. This new point will either be $\polyVec{\gamma(1)}{2}$, if $\gamma(0)=\gamma(1)$, or $\polyVec{\gamma(1)}{1}$, if $\gamma(0)\neq\gamma(1)$. We continue in this way, increasing the `heights' of the vertices according to the values of $\gamma$, until we reach all the maximum heights $N_0,\ldots,N_s$ and have triangulated the whole of $\tilde{\Pi}$.\end{rem}

\begin{eg}\label{eg.2foldtriangulation}We return to the very simple example of \ref{eg.3foldfan}, the conifold. The subvariety corresponding to the partition $\Gamma=(0 1)$ is the $A_1$ surface of Section~\ref{sect.thesurfaces}, and the fan is the cone on the polytope in Figure~\ref{fig:polytopesForSubvarieties}\ref{2foldfan}.
\end{eg}

\begin{eg}\label{eg.triangulation}
 Let $k=4$, so that $X$ is a $6$-fold, and let $\Gamma$ be the partition as in Example~\ref{eg.subvariety_coords}. Then $s=1$, so $X_\Gamma$ is a $3$-fold. The  polytope $\tilde{\Pi}$ lies in the affine subspace
$$\{\tilde{u} + \tilde{v}= 2\tilde{z}_0 + 3\tilde{z}_1, \;\;\;\;\tilde{z}_0+\tilde{z_1} = 1\}\subset \C^4.$$
We can use $(\tilde{z_1}, \tilde{u})$ as co-ordinates on this subspace, then we can draw the triangulation of $\tilde{\Pi}$ (see Figure~\ref{fig:polytopesForSubvarieties}\ref{3foldfan2}).
%%%%%% BEGIN TikZ
\begin{figure}[h]
\begin{center}
                  \begin{minipage}[b]{.1\linewidth}
                  \end{minipage}%
                  %%%%%%%%% subfigure
                  \begin{minipage}[b]{.4\linewidth}
                    \centering
\begin{tikzpicture}[scale=1.2]
\draw[->] (0,0) -- (0.5,0) node[right]{\scriptsize $z_1$};
\draw[->] (0,0) -- (0,2.4) node[above]{\scriptsize $u$};
\draw[help lines] (0,0) grid (0.5,2.4);

% polytope
\draw[line width=1.5pt]
	(0,2) node[left] {\scriptsize $\polyVec{0}{2}$} --
	(0,1) node[left] {\scriptsize $\polyVec{0}{1}$} --
	(0,0) node[left] {\scriptsize $\polyVec{0}{0}$};

% triangulation
\node at (0.2,1.5) {\scriptsize $\tilde{\Delta}_1$};
\node at (0.2,0.5) {\scriptsize $\tilde{\Delta}_0$};

\end{tikzpicture}
\vspace{0.7cm}
                    \subcaption{Triangulation for $k=1$, $\Gamma=(0 1)$. \\ The associated $X_\Gamma$ is the $A_1$ surface.}\label{2foldfan}
                  \end{minipage}%
                  %%%%%%%%% subfigure
                  \begin{minipage}[b]{.4\linewidth}
                    \centering
\begin{tikzpicture}[scale=1.2]
\draw[->] (0,0) -- (1.99,0) node[right]{\scriptsize $\tilde{z_1}$};
\draw[->] (0,0) -- (0, 3.4) node[above]{\scriptsize $\tilde{u}$};
\draw[help lines] (0,0) grid (1.99, 3.4);

% polytope
\draw[line width=1.5pt]
	(0,0) node[left]{\scriptsize $\polyVec{0}{0}$} -- 
	(0,1) node[left]{\scriptsize $\polyVec{0}{1}$} --
	(0,2) node[left]{\scriptsize $\polyVec{0}{2}$} --
	(1,3) node[right]{\scriptsize $\polyVec{1}{3}$} --
	(1,2) node[right]{\scriptsize $\polyVec{1}{2}$} --
	(1,1) node[right]{\scriptsize $\polyVec{1}{1}$}--
	(1,0) node[right]{\scriptsize $\polyVec{1}{0}$}-- cycle;

% triangulation
\draw[triangulation] (1,0) -- (0,1);
\draw[triangulation] (1,0) -- (0,2);
\draw[triangulation] (1,1) -- (0,2);
\draw[triangulation] (1,2) -- (0,2);
\node at (1/3,1/3) {\scriptsize $\tilde{\Delta}_0$};
\node at (1/3,1) {\scriptsize $\tilde{\Delta}_1$};
\node at (3/4,1) {\scriptsize $\tilde{\Delta}_2$};
\node at (2/3,5/3) {\scriptsize $\tilde{\Delta}_3$};
\node at (2/3,7/3) {\scriptsize $\tilde{\Delta}_4$};

\end{tikzpicture}
\vspace{0.3cm}
                    \subcaption{Triangulation for $k=4$, $\Gamma=(0 1)(2 3 4)$. \\ In this case, $X_\Gamma$ is a $3$-fold.}\label{3foldfan2}
                  \end{minipage}
                  \caption{Triangulations of polytope $\tilde{\Pi}$ associated with subvarieties $X_\Gamma$.} \label{fig:polytopesForSubvarieties}
\end{center}
                \end{figure}
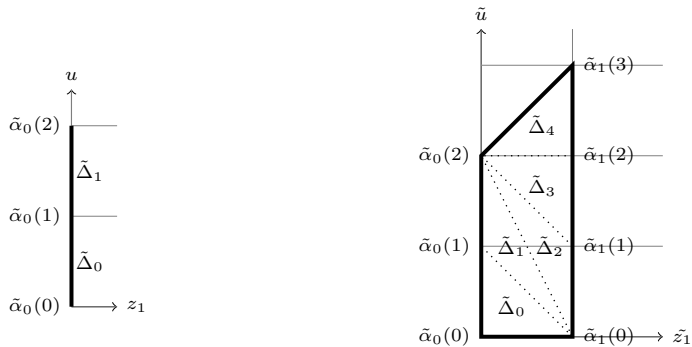
%%%%%% END TikZ
\end{eg}

\subsubsection{GIT description}
\label{sect.subvarietyGITdescrip}

We've constructed the toric data for $X_\Gamma$, however we claim that $X_\Gamma$ is in fact a GIT quotient of a vector space by a torus, which for us is a more useful description. (Of course this is not surprising, as it's true of `most' toric varieties.) We'll first argue this abstractly, then we'll explicitly identify the GIT data.

If we forget the individual cones in the fan (i.e. the triangulation of $\tilde{\Pi}$), then we have only the data of the set of $(k+s+2)$ vectors $\polyVec{t}{\delta}$ that generate the rays. Let $\set{\polyVecSymbol{t}{\delta}}$ be an abstract set bijecting with the set of vectors $\{\polyVec{t}{\delta}\}$. Then we have a map between lattices
\begin{align*} \tilde{P}: \Z^{k+s+2} := \langle \polyVecSymbol{t}{\delta} \rangle_{\Z} & \to \Z^{2+(s+1)}, \\
\polyVecSymbol{t}{\delta} & \mapsto \polyVec{t}{\delta}.
\end{align*}
 The kernel of this map is a rank $k$ lattice, and the associated torus $\tilde{T} = (\C^*)^k$  acts on the vector space $\tilde{V}$ generated by the $\polyVecSymbol{t}{\delta}$. This is a GIT problem, and the possible phases $\tilde{V}\sslash \tilde{T}$ correspond to particular subdivisions of the polytope $\tilde{\Pi}$. 

\begin{prop}\label{prop.XGammaisaGITquotient} $X_\Gamma$ is one of the phases of the GIT problem $\tilde{V} \sslash \tilde{T}$.
\end{prop}
\begin{proof}Our choice of character $\theta$ specifies a line bundle $L$ on $X$, making $X$ a projective-over-affine variety. By definition, $X_\Gamma$ is also projective-over-affine: it is $\operatorname{Proj}$ of the graded ring of sections of powers of the line bundle $\tilde{L} = L|_{X_\Gamma}$. The line bundle $\tilde{L}$ is necessarily toric, so it corresponds to some character $\tilde{\theta}$ of the torus $(\C^*)^k$, i.e. a stability condition for our new GIT problem. The corresponding GIT quotient is by definition $X_\Gamma$. 
\end{proof}

We'll refer to $X_\Gamma$ as the `standard phase' for this GIT problem. There are some other obvious phases: if we start with a non-standard phase $X^\sigma$ for the ambient space (associated to some $\sigma\in S_{k+1}$), and then impose the moment map equations corresponding to $\Gamma$, we get another toric variety $X_\Gamma^\sigma$ which is birational to $X_\Gamma$. It's clear how to get the toric fan for $X_\Gamma^\sigma$: we just apply the permutation $\sigma$ before running the recipe from the previous section. Consequently $X_\Gamma^\sigma$ is constructed by the same GIT problem that constructs $X_\Gamma$, so it's another phase.

 From the recipe for the toric fans, it's clear that if $\sigma$ fixes the partition $\gamma$ then $X_\Gamma$ and $X_\Gamma^\sigma$ are identical (i.e. isomorphic over the affine base).  Therefore these phases depend only on the coset $\sigma S_\Gamma$, where $S_\Gamma\subset S_{k+1}$ is the symmetry group of $\Gamma$.

 As we will see later, these are not the only phases of this GIT problem: there are phases which are orbifolds, and so cannot arise in this way.
\pgap

Next we'll explicitly identify the kernel of $\tilde{P}$, i.e. the action of $\tilde{T}=(\C^*)^k$ on $\tilde{V}$. Note that for any $t$, and any $\delta\in [0, N_t-1]$, we have 
$$\polyVecSymbol{t}{\delta+1} - \polyVecSymbol{t}{\delta} \overset{\tilde{P}}{\mapsto} (-1, 1, 0,0,\ldots,0),$$
so if we pick any two of these elements then their difference
$$\polyVecSymbol{t}{\delta+1} - \polyVecSymbol{t}{\delta} - \polyVecSymbol{t'}{\delta'+1} + \polyVecSymbol{t'}{\delta'} $$
lies in the kernel of $\tilde{P}$, and it's evident that the whole kernel is spanned by vectors of this form. 

We can write down a convenient basis for $\ker{\tilde{P}}$ by considering pairs of adjacent cones in the toric data. For example, consider the two simplices $\tilde{\Delta}_{0}$ and $\tilde{\Delta}_1$. There are two situations to consider:
\begin{itemize}
\item If $\gamma(0)\neq \gamma(1)$, then the four vectors
$$\polyVec{\gamma(0)}{1},\;\; \polyVec{\gamma(0)}{0}, \;\; \polyVec{\gamma(1)}{1}, \;\; \polyVec{\gamma(1)}{0} \;\;\in\;\; \tilde{\Delta}_0\cup \tilde{\Delta}_1 $$
are coplanar, and we let
$$\tau_1 \;=\; \polyVecSymbol{\gamma(0)}{1}- \polyVecSymbol{\gamma(0)}{0} - \polyVecSymbol{\gamma(1)}{1} + \polyVecSymbol{\gamma(1)}{0} \;\in\; \Z^{k+s+2}$$
which lies in the kernel of $\tilde{P}$.
\item Alternatively, if $\gamma(0)=\gamma(1)$ then the three vectors
$$\polyVec{\gamma(0)}{2},\;\; \polyVec{\gamma(0)}{1}, \;\;\polyVec{\gamma(0)}{0} \;\;\in\;\; \tilde{\Delta}_0\cup \tilde{\Delta}_1 $$
are collinear, and we let 
$$\tau_1 \;= \;  2 \polyVecSymbol{\gamma(0)}{1}  - \polyVecSymbol{\gamma(0)}{2} - \polyVecSymbol{\gamma(0)}{0} \;\in\; \Z^{k+s+2},$$
which also lies in the kernel of $\tilde{P}$.
\end{itemize}
In general, we have the following proposition.

\begin{prop}\label{prop.kernelofPtilde}For $i\in [1,k]$, two adjacent simplices $\tilde{\Delta}_{i-1}$ and $\tilde{\Delta}_i$ give a vector
$$\tau_i \;=\; \polyVecSymbol{\gamma(i-1)}{n^{i-1}_{\gamma(i-1)}+1}
- \polyVecSymbol{\gamma(i-1)}{n^{i-1}_{\gamma(i-1)}} 
- \polyVecSymbol{\gamma(i)}{n^i_{\gamma(i)}+1} 
+ \polyVecSymbol{\gamma(i)}{n^i_{\gamma(i)}}    $$
in the kernel of $\tilde{P}$, and this set of vectors forms a basis. 
\end{prop}

\begin{eg}\label{eg.2foldweights}For example \ref{eg.2foldtriangulation}, the $A_1$ surface, the above Proposition~\ref{prop.kernelofPtilde} gives a single generator
$$\tau_1 \;=\; 2\polyVecSymbol{0}{1}-\polyVecSymbol{0}{0} -  \polyVecSymbol{0}{2}$$
for the kernel. This tells us that $X_\Gamma$ is a GIT quotient of $\C^3$ by $\C^*$ acting with weights $(-1,2,-1)$, which of course we knew. 
\end{eg}

\begin{eg}Returning to the setting of Example~\ref{eg.triangulation}, we find that the two simplices $\tilde{\Delta}_1$ and $\tilde{\Delta}_2$ yield coplanar vectors as in case (1) above, and the other pairs of adjacent simplices give collinear vectors as in case (2). So this $3$-fold $X_\Gamma$ is a GIT quotient of $\C^7$ by the torus $(\C^*)^4$ acting with weights as follows:
$$\begin{pmatrix}-1 &2 & -1 & 0 & 0 & 0 & 0 \\
                    0 &-1 &1 &1 & -1 & 0 & 0 \\
                    0 & 0 & 0 & -1 &2 & -1 & 0 \\
                    0 & 0 & 0 & 0 & -1 &2 & -1 \end{pmatrix} $$

\begin{comment}
namely
$$\polyVec{0}{1},\;\; \polyVec{0}{2}, \;\; \polyVec{1}{1}, \;\; \polyVec{1}{2} \;\;\in \tilde{\Delta}_1\cup \tilde{\Delta}_2 $$
and so $$\tau_2 \;=\; \polyVec{0}{1}- \polyVec{0}{2} - \polyVec{1}{1} + \polyVec{1}{2} \in \Z^{k+s+2}$$
The other pairs of simplices proceed as in case (2) above, yielding
\begin{align*}
\tau_1 & \;=\; \polyVec{0}{0}- 2 \polyVec{0}{1} - \polyVec{0}{2} \\
\tau_3 & \;=\; \polyVec{1}{0}- 2 \polyVec{1}{1} - \polyVec{1}{2} \\
\tau_4 & \;=\; \polyVec{1}{1}- 2 \polyVec{1}{2} - \polyVec{1}{3}
\end{align*}
as vectors in $\Z^{k+s+2}$.
\end{comment}

\end{eg}
\pgap

As we explained in the proof of Proposition~\ref{prop.XGammaisaGITquotient}, our choice of character $\theta$ when constructing the ambient space $X$ induces a character $\tilde{\theta}$ of $\tilde{T}$. We'll now identify this character explicitly. This will tell us what stability conditions we can choose on our GIT problem $\tilde{V} \sslash \tilde{T}$ to produce the standard phase $X_\Gamma$.

Recall that our standard $\vartheta$ extends to a concave PL function on the toric fan for $X$. We just restrict this function to the fan for $X_\Gamma$, then its values on the vertices $\polyVec{t}{\delta}$ give a lift $\tilde{\vartheta}$ of the character $\tilde{\theta}$. 
The vertex $\polyVec{t}{\delta}$ can be written as a sum 
$$\polyVec{t}{\delta} = \sum_{i\in D} \polyTopVec{i} + \sum_{j\in \gamma^{-1}(t)-D} \polyBaseVec{j} $$
for any subset $D\subset \gamma^{-1}(t)$ of size $\delta$ (this is immediate from the definition \eqref{eq.polyVecDef}). Because of our choice \eqref{thetaform} of $\vartheta$ we have that $\tilde{\vartheta}(\polyBaseVec{i})=0$, so the fact that the function $\vartheta$ is concave implies that
\begin{equation}\label{eqn.tildevartheta}
\tilde{\vartheta}: \polyVec{t}{\delta} \mapsto \max_{\substack{D\subset\gamma^{-1}(t)\\|D| = \delta}} \sum_{d\in D} \vartheta_d,
\end{equation}
i.e. the sum of the $\delta$ biggest values lying in the $t^{\text{th}}$ part of the partition.
 If we restrict this function $\tilde{\vartheta}$ to the kernel of $\tilde{P}$, this will give us our character $\tilde{\theta}$ of the torus $(\C^*)^k$.   To see what the result is, note first that
$$\tilde{\vartheta} \left(\polyVec{t}{\delta+1}-\polyVec{t}{\delta}\right)\; =\; \vartheta_d $$
where $\vartheta_d$ is the ${(\delta+1)}^{\text{th}}$ biggest value lying in the $t^{\text{th}}$ part of the partition. Consequently
$$\tilde{\vartheta}\left( \polyVec{\gamma(i)}{n_{\gamma(i)}^{i}+1}
-\polyVec{\gamma(i)}{n_{\gamma(i)}^{i}}\right)\; =\; \vartheta_i $$
and therefore
\begin{equation}\label{eqn.tildetheta}\tilde{\theta}(\tau_i) =  \vartheta_{i-1} - \vartheta_i  = \theta_i.\end{equation}
So in this basis, setting all coefficients of the character $\tilde{\theta}$ to be positive produces the standard phase $X_\Gamma$. In fact we will see in the next section that this is precisely the chamber of characters that produce this phase, so we'll refer to it as the \textit{standard chamber} for this GIT problem.

\subsection{Families of rational curves}
\label{sect.familiesofrationalcurves}

In this section we'll make some very straightforward observations on the geometry of $X$ and its subvarieties $X_\Gamma$.

 Recall that the fan for $X$ is the union of the cones on the simplices $\Delta_i$ \eqref{Deltai}. Each such cone corresponds to a Zariski open set $U_i\subset X$, which is isomorphic to $\C^{k+2}$ with co-ordinates $a_0,\ldots,a_i, b_i, \ldots, b_k$. We can also think of $U_i$ as (the image of) the affine subspace in $V$ where we set all the remaining co-ordinates to 1. 
 
 If we take two neighbouring simplices $\Delta_{i-1}$ and $\Delta_i$, and consider the corresponding open sets $U_{i-1}$ and $U_i$, we find that they glue together to give an open set
$$\cN_i := U_{i-1}\cup U_i \cong \cO(-1)^{\oplus 2}_{\P^1} \times \C^{k-1}.$$
The ratio $a_i:b_{i-1}$ gives co-ordinates on the $\P^1$, the co-ordinates $a_{i-1}$ and $b_i$ are the fibre directions on the bundle, and the remaining co-ordinates $a_0,..,a_{i-2},b_{i+1},\ldots,b_k$ parametrize the $\C^{k-1}$.
 We will be interested in the codimension 2 subvarieties 
\begin{equation}\label{eq.twistSubvarieties}S_i = \set{a_{i-1}=b_i = 0} \subset X \end{equation}
corresponding to the zero section of the bundle, where here $1\leq i \leq k$. Each $S_i$ lies entirely within the open set $\cN_i$, so it's a trivial family of rational curves, and it has this very simple Zariski neighbourhood. 

\begin{rem}\label{rem.destabilizedSub}
If we vary our GIT quotient from the standard chamber \eqref{standardtheta} to the neighbouring chamber
\begin{equation}\label{neighbouringchamber}\vartheta_0>\ldots
>\vartheta_{i-2}>\vartheta_{i}>\vartheta_{i-1}
>\vartheta_{i+1}>\ldots>\vartheta_k \end{equation}
by passing through the wall $\theta_i=0$,
then $S_i$ is exactly the locus that becomes unstable. To see this, observe that a cone disappears from the fan if and only if it contains the cone on the interval $[\polyTopVec{i-1}, \polyBaseVec{i}]$, which is the cone corresponding to the toric subvariety $S_i$. See Figure~\ref{fig:flops}\ref{3foldflops} for an example.
\end{rem}

%%%%%% BEGIN TikZ
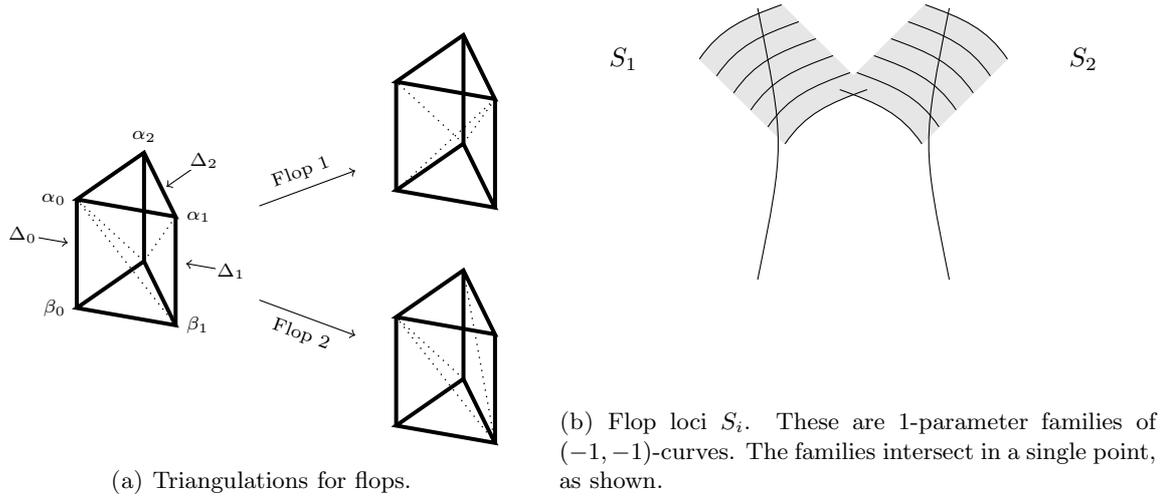
\begin{figure}[h]
                  \begin{minipage}[b]{.5\linewidth}
                    \centering
\begin{tikzpicture}[scale=1.2]

\def\flopxshift{3.5}
\def\flopyshift{1.3}

% flopping polytope
\polytopeStdTriangulation{\angleA}{\scaleA}{\angleB}{\scaleB}{0}{0}{yes}{basic}{\scaleC};
% two flops
\polytope{\angleA}{\scaleA}{\angleB}{\scaleB}{\flopxshift}{\flopyshift}{no}{flop1}{\scaleC};
\polytope{\angleA}{\scaleA}{\angleB}{\scaleB}{\flopxshift}{-\flopyshift}{no}{flop2}{\scaleC};

% flopped triangulations
% flop1
\draw[triangulation] (flop1B0) -- (flop1A1);
\draw[triangulation] (flop1B2) -- (flop1A0);
\draw[triangulation] (flop1B2) -- (flop1A1);
% flop2
\draw[triangulation] (flop2B1) -- (flop2A0);
\draw[triangulation] (flop2B2) -- (flop2A0);
\draw[triangulation] (flop2B1) -- (flop2A2);

% arrows for flops
\draw[->] ($(basicCentre)!0.4!(flop1Centre)$) -- node[above,sloped]{\scriptsize Flop $1$} ($(basicCentre)!0.7!(flop1Centre)$);
\draw[->] ($(basicCentre)!0.4!(flop2Centre)$) -- node[below,sloped]{\scriptsize Flop $2$} ($(basicCentre)!0.7!(flop2Centre)$);
\end{tikzpicture}
                    \subcaption{Triangulations for flops. }\label{3foldflops}
                  \end{minipage}%
                  \begin{minipage}[b]{.5\linewidth}
                    \centering
\begin{tikzpicture}[scale=1.8]

\def\halfwidth{0.7}
\def\pinch{0.2}
% floc loci
\def\intersection{0.3}
\def\crossing{0.1}
\def\curve{0}
\def\familysep{0.125}
\def\numcurves{5}

% shading
\draw[color=\defBGColor,fill=\defBGColor] \defCurveCoords{+1}{0} -- \defCurveCoordsRev{+1}{\familysep*\numcurves} -- cycle;
\draw[color=\defBGColor,fill=\defBGColor] \defCurveCoords{-1}{0} -- \defCurveCoordsRev{-1}{\familysep*\numcurves} -- cycle;

% family
\foreach \i in {0,...,\numcurves} {
	\draw \defCurveCoords{+1}{\familysep*\i};
	\draw \defCurveCoords{-1}{\familysep*\i};
	
}

\coordinate (S1) at (-\halfwidth+\pinch-\familysep*\numcurves,\familysep*\numcurves);
\node[left of=S1] {$S_1$};

\coordinate (S2) at (+\halfwidth-\pinch+\familysep*\numcurves,\familysep*\numcurves);
\node[right of=S2] {$S_2$};

% surface
\draw (-\halfwidth,1) .. controls (-\halfwidth+\pinch,0) .. (-\halfwidth,-1);
\draw (+\halfwidth,1) .. controls (+\halfwidth-\pinch,0) .. (+\halfwidth,-1);

\end{tikzpicture}
\vspace{1.5cm}
                    \subcaption{Flop loci $S_i$. These are $1$-parameter families of $(-1,-1)$-curves. The families intersect in a single point, as shown.}\label{fig.flopsloci}
                  \end{minipage}
                  \caption{Flop geometry for $k=2$, for $4$-fold $X$.}\label{fig:flops}
                \end{figure}
%%%%%% END TikZ

\subsubsection{Subvarieties $X_\Gamma$}
Now we pick a partition $\Gamma$ with $s$ pieces. We want to find the intersection of these $S_i$ with the subvariety $X_\Gamma$, which is cut out by some subset of the equations $z_j - z_l=0$. Within $\cN_i$, the global functions $z_j$ become 
$$z_j = \left\{\begin{array}{ll} a_j & j\leq i-2, \\
a_j b_j &  i\leq j \leq i+1,\\
b_j & j\geq i+1.
\end{array}\right. $$ 
There are two cases to distinguish:
\begin{list}{(\alph{enumi})}{\usecounter{enumi}}
\item If $i$ and $i-1$ are not in the same piece of $\Gamma$, so we don't impose $z_{i-1}=z_i$, then we have
 $$\tilde{\cN}_i := \cN_i \cap X_\Gamma \cong \cO(-1)^{\oplus 2}_{\P^1} \times \C^{s-1}. $$
The subvarieties $S_i$ and $X_\Gamma$ intersect transversely, and the subvariety
\begin{equation}\label{eqn.twistSubvarietiesNonVersal}\tilde{S}_i := S_i\cap X_\Gamma\end{equation}
 is an $(s-1)$-parameter trivial family of rational $(-1,-1)$-curves.
\item If $i$ and $i-1$ are in the same piece of $\Gamma$ then
$$\tilde{\cN}_i = \cN_i \cap X_\Gamma \cong \cO(-2)_{\P^1} \times \C^s. $$
In this case the intersection of $S_i$ and $X_{\Gamma}$ is not transverse, and $\tilde{S}_i$ is an $s$-parameter trivial family of rational $(-2)$-curves. 
\end{list}
In each case, $\tilde{S}_i$ has this very simple Zariski neighbourhood $\tilde{\cN}_i $.

\begin{figure}[h]
\begin{center}
\begin{tikzpicture}[scale=1.5]

\def\halfwidth{0.7}
\def\pinch{0.2}

% floc loci
\def\intersection{0.3}
\def\crossing{0.1}
\def\curve{0}
\def\familysep{0.125}
\def\numcurves{5}

% shading
\draw[color=\defBGColor,fill=\defBGColor] \defCurveCoords{+1}{0} -- \defCurveCoordsRev{+1}{\familysep*\numcurves} -- cycle;

\foreach \i in {0,...,\numcurves} {
	\draw \defCurveCoords{+1}{\familysep*\i};

}

\def\i{0}
	\draw \defCurveCoords{-1}{\familysep*\i};

\coordinate (S1) at (-\halfwidth+\pinch-\familysep*\numcurves,\familysep*\numcurves);
\node[left of=S1] {$S_1$};

\coordinate (S2) at (+\halfwidth-\pinch,0);
\node[right of=S2] {$S_2$};

% surface
\draw (-\halfwidth,1) .. controls (-\halfwidth+\pinch,0) .. (-\halfwidth,-1);
\draw (+\halfwidth,1) .. controls (+\halfwidth-\pinch,0) .. (+\halfwidth,-1);

\end{tikzpicture}
\end{center}
\caption{Flop loci for $k=2$, partition $\Gamma = (0 1)(2)$. In this case $s=1$, and $X_\Gamma$ is a $3$-fold. The locus $S_1$ is a $1$-parameter families of $(-2)$-curves, and $S_2$ is a single $(-1,-1)$-curve. Compare Figure~\ref{fig:flops}\ref{fig.flopsloci}.}\label{floploci}
\end{figure}
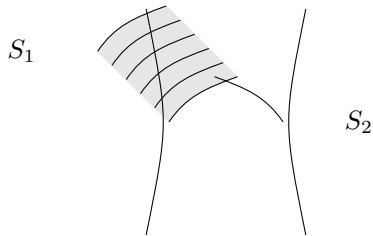

We can also find these $\tilde{S}_i$ in the toric data for $X_\Gamma$, or equivalently in terms of the GIT problem $\tilde{V} \sslash \tilde{T}$.  Recall that the toric fan for $X_\Gamma$ is the union of the cones on the simplices $\tilde{\Delta}_i$ \eqref{tildeDeltai}. Each such cone corresponds to a `toric chart' $\tilde{U}_i\subset X_\Gamma$, isomorphic to affine space, and $\tilde{\cN}_i = \tilde{U}_{i-1}\cup \tilde{U}_i$, so $\tilde{\cN}_i$ is a toric variety whose fan has only two cones. 
If we inspect this toric data, we see that it's constructing $\tilde{\cN}_i$ as a GIT quotient
$$\tilde{\cN}_i = \tilde{V}_i \sslash_{\tilde{\theta}}\, \tau_i $$
where $\tilde{V}_i\subset \tilde{V}$ is the subspace spanned by the vertices that appear in  $\tilde{\Delta}_{i-1}\cup \tilde{\Delta}_i$, and $\tau_i\subset \tilde{T}$ is the 1-parameter subgroup identified in Proposition~\ref{prop.kernelofPtilde}. The torus action is trivial in some directions (these contribute the trivial directions in $\tilde{\cN}_i$), and in the remaining directions it's the usual GIT problem from Section~\ref{sect.A1case} constructing either (a) $\cO(-1)^{\oplus 2}_{\P^1}$ or (b) $\cO(-2)_{\P^1}$. 

By \eqref{eqn.tildetheta}, the value of a standard stability condition $\tilde{\theta}$ on the 1-parameter subgroup $\tau_i$ is  $( \vartheta_{i-1} - \vartheta_i)$. Consequently, when we leave the standard chamber across the wall corresponding to $\vartheta_i=\vartheta_{i-1}$ the family of curves $\tilde{S}_i$ becomes unstable. What happens when we cross the wall differs in the two cases:
\begin{list}{(\alph{enumi})}{\usecounter{enumi}}
\item Across the wall, there is a neighbouring phase for $X_\Gamma$ where the locus $\tilde{S}_i$ has been flopped: locally around $\tilde{S}_i$, this is just a trivial family of standard $3$-fold flops. This neighbouring phase is a subvariety inside a non-standard phase of the ambient space $X$: we move to the chamber \eqref{neighbouringchamber} and then impose moment maps corresponding to the same $\Gamma$. In other words, this flop of the locus $\tilde{S}_i$ is just induced by a flop of the locus $S_i$ inside $X$.
\item Across the wall there is a neighbouring phase for $X_\Gamma$ where the locus $\tilde{S}_i$ has been removed, and replaced with a trivial family of orbifold points with $\Z_2$ isotropy groups. This neighbouring phase is an orbifold, and so cannot be a subvariety inside any phase of $X$. The characters $\tilde{\theta}$ in this neighbouring chamber do not arise from any $\theta$.
\end{list}
If we travel further away from the standard chamber then we can reach other kinds of phases, which can have larger isotropy groups. Note also that the above description still works if we start in any phase  $X_\Gamma^\sigma$ induced from a non-standard phase $X^\sigma$ of the ambient space: we just have to permute the variables.

\subsection{Derived equivalences from wall-crossing}\label{sect.derivedequivalencesfromwalls}

In this section we'll discuss the derived equivalences that correspond to the wall-crossing described in the previous section, following the general theory of \cite{DHL, BFK} and especially \cite{HLS}.\footnote{The general theory deals with GIT quotients under arbitrary reductive groups,  since we only care about the abelian case we'll only be using a small part of it.}

 Start with a standard phase $X_\Gamma$ of one our subvarieties, pick a value of $i$, and let $X_\Gamma '$ be the phase obtained by crossing the wall corresponding to $\vartheta_i=\vartheta_{i-1}$. As we saw in Section~\ref{sect.familiesofrationalcurves}, the birational transformation between $X_\Gamma$ and $X_\Gamma '$ is locally just a trivial family of the two kinds of $A_1$ flops discussed in Section~\ref{sect.A1case}. Consequently, there are $\Z$-many derived equivalences
$$\psi_k : \Db(X_\Gamma) \isoto \Db(X_\Gamma ') $$
which roughly-speaking come from family versions of the derived equivalences we saw in the $A_1$ case. Let's explain this more precisely.

\subsubsection{General theory}\label{sect.generaltheoryforwindows}
 Let $Y$ be a smooth projective-over-affine variety equipped with a $\C^*$ action, and let
$$\sigma: Z \into Y$$
be the inclusion of the fixed locus. For simplicity, assume that $Z$ is connected. To form a GIT quotient, we need to choose an equivariant line bundle on $Y$, so let's choose $\cO_Y(1)$ (i.e. the trivial line bundle equipped with a weight 1 action). Then the unstable locus is the subvariety $S_+\subset Y$ of points that flow to $Z$ under the $\C^*$ action (as $t\in \C^*$ goes to 0), so the semistable locus is $Y^{ss}:= Y - S_+$, and the GIT quotient is $[Y^{ss} / \C^*]$.\footnote{Purists may insist that the GIT quotient is really the coarse moduli space of this stack, but this abuse-of-language is becoming quite common.}

 Now let $\eta_+$ be the $\C^*$-weight of the line bundle $\det N^\vee_{S_+} Y$ along $Z$, which is necessarily a positive integer. For any choice of integer $k$, we define a `window'
$$\cW(k) \subset \Db(\quotstack{Y}{\C^*}) $$
to be the full subcategory of objects $\cE$ such that all homology sheaves of $L\sigma^* \cE$ 
have $\C^*$-weights lying in the interval
$$[k,\; k+\eta_+).$$
So for example, the equivariant line bundle $\cO_Y(i)$ lies in $\cW(k)$ if and only if $i$ lies in the above interval. In fact if $Y$ is a vector space then we may equivalently define $\cW(k)$ to be the subcategory generated by this set of line bundles, as we did in Section~\ref{sect.A1case}. A basic result of the theory cited above is that the restriction map
$$\cW(k) \to \Db(\quotstack{Y^{ss}}{\C^*}) $$
is an equivalence (for any $k$). If we change our stability condition to $\cO_Y(-1)$ then the unstable locus changes to the subvariety $S_-$ of points that flow away from $Z$ (i.e. they flow to $Z$ as $t\in\C^*$ goes to infinity), and we have a similar description of the derived category of the other GIT quotient $\quotstack{(Y - S_-)}{\C^*}$ based on the weight $\eta_-$ of $(\det N_{S_-} Y) |_Z$. In particular if we happen to have $\eta_-=\eta_+$ then the two GIT quotients are derived equivalent, and each choice of window gives a specific derived equivalence $\psi_k$.

\subsubsection{Our examples}\label{sect.wallcrossinginourexamples}
 Now let's apply this theory to the wall-crossing between the phases $X_\Gamma$ and $X_\Gamma '$. We saw in Section~\ref{sect.subvarietyGITdescrip} that both are GIT quotients of a vector space $\tilde{V}$ by a torus $\tilde{T}$. The wall between the phases is part of the annihilator of the 1-parameter subgroup $\tau_i \subset \tilde{T}$. Choose a stability condition that lies exactly on the wall (but not also on any other walls), and let $\tilde{V}^{ss}$ be the semistable locus for this stability condition. If we pick a splitting $\tilde{T} = \tau_i\oplus \tau_i^\perp$, then the torus $\tau_i^\perp$ acts freely on $\tilde{V}^{ss}$, the quotient $Y=\tilde{V}^{ss}/\tau_i^\perp$ is a smooth projective-over-affine variety, and we have an isomorphism of stacks
$$\quotstackBig{\tilde{V}^{ss}}{\tilde{T}} = \quotstackBig{Y}{\tau_i}.$$
Our phases $X_\Gamma$ and $X_\Gamma '$ are the two possible GIT quotients of $Y$ by $\tau_i$. By definition the subvariety $S_-\subset Y$ is the closure of the locus that becomes unstable when we cross the wall, so we know from Remark~\ref{rem.destabilizedSub} that $S_-$ is (the closure in $Y$ of)  $\tilde{S}_i$, our trivial family of rational curves \eqref{eq.twistSubvarieties}. Similarly $S_+$ is (the closure of) the subset that replaces $\tilde{S}_i$ after the flop: by Section~\ref{sect.familiesofrationalcurves} it's either (a) another family of rational curves, or (b) a family of orbifold points.

Now we explain the calculation of the numerical data $\eta_{\pm}$ in our examples. We can calculate in the Zariski open neighbourhood $\quotstackinline{\tilde{V}_i}{\tau_i}$, i.e. the GIT problem that constructs the neighbourhood $\tilde{N}_i$. The action of $\tau_i$ is trivial in most directions, so we can reduce to one of our two basic $A_1$ examples from Section~\ref{sect.A1case}. Then it's easy to calculate\footnote{In fact the equality of $\eta_-$ and $\eta_+$, though not their value, follows from the fact that both sides are Calabi--Yau.} that
$$\eta_- = \eta_+ = 2 $$
and so the general theory tells us that we have derived equivalences $\psi_k$ between the two phases, defined using windows $$\phantom{{}^{9}}\cW(k)\subset \Db\left(\quotstack{\tilde{V}^{ss}}{\tilde{T}}\right).$$

\subsubsection{Geometric description of equivalences}
\label{section.geom_descrip_equiv}
The family version of Proposition~\ref{prop.windowisflop} holds true \cite[Section 3.1]{HLS}, so the equivalence
$$\psi_{-1}:\Db(X_\Gamma) \to \Db(X_\Gamma ')$$
can also be described using the common birational roof. 

Similarly, the autoequivalence
$$(\psi_{-1})^{-1}\circ \psi_0 : \Db(X_\Gamma) \to \Db(X_\Gamma)$$
is a family version of the spherical twist around $\cO_{\P^1}(-1)$ that we discussed in Section~\ref{sect.A1case}.  More precisely, let
$$\begin{tikzcd} \tilde{S}_i \rar{\iota}\dar[swap]{\mu} & X_\Gamma\\ 
\C^{r} & \end{tikzcd} $$
be the trivial family of rational curves (here $r$ is either $(s-2)$ or $(s-1)$ depending on which case we're in), and let  $\cO(-1)$ be the relative tautological line bundle for $\mu$. Then we define a functor $\mathcal{S}$ with right adjoint $\mathcal{R}$, as follows:
\begin{equation*}
\begin{tikzpicture}[bend angle=15]
\node (a1) at (0,0) {$\Db(\C^r)$};
\node (a2) at (4,0) {$\Db(X_\Gamma^\sigma)$};
\draw[->,bend left] (a1) to node[above] {$\mathcal{S}$} (a2);
\draw[->,bend left] (a2) to node[below] {$\mathcal{R}$} (a1);
\end{tikzpicture}
\end{equation*}
\begin{align*}\mathcal{S}(\cE) &=\iota_* (\mu^*\cE\otimes\cO(-1))\\
\mathcal{R}(\cF ) &= \mu_*(\iota^!\cF\otimes \cO(1))\end{align*}
The functor $\mathcal{S}$ is a spherical functor in the sense of \cite{Anno,AnnoLogv}, and we can define an associated spherical twist functor
\begin{equation}\label{eqn.ST_family_twist}T_{\mathcal{S}}: \mathcal{E} \mapsto  \cone{\mathcal{S}\mathcal{R}\cE \to \cE}. \end{equation}
Recall that $\tilde{S}_i$ has a nice Zariski neighbourhood $\tilde{\cN}_i$ which is a trivial family of copies of either (a) $\cO(-1,-1)$ or (b) $\cO(-2)$ over $\P^1$. Within this neighbourhood $T_{\mathcal{S}}$ is just a trivial family of ordinary spherical twists around $\cO_{\P^1}(-1)$, and outside $\tilde{\cN}_i$ it's just the identity functor. By \cite[Section 3.2]{HLS}, we have that
$$ (\psi_{-1})^{-1}\circ \psi_0 = T_{\cS}.$$
Again this discussion holds equally well if we start our wall-crossing from any phase $X_\Gamma^\sigma$.

\section{Fayet--Iliopoulos parameter spaces}
\label{sect.FIps}
\subsection{Toric Mirror Symmetry heuristics}
\label{sect.MSheuristics}

Choose an action of a rank $r$ torus $T$ on an $n$-dimensional vector space $V$. After picking diagonal co-ordinates, the action is given by the matrix of weights
$$Q: \Z^{r} \to \Z^n. $$
We let 
$$P: \Z^{n} \to \Z^{n-r} $$
be the cokernel of $Q$ (modulo torsion). Notice that $\Z^{n}$ has a canonical positive orthant (since it comes from a vector space), so it has a canonical set of generators, up to permutation.
If we were to pick a character $\theta$ of $T$ and construct the corresponding toric variety $X = V \sslash_\theta T$, then the toric fan for $X$ lies in the lattice $\Z^{n-r}$, and the rays in the fan are (some subset of) the images of the canonical generators under $P$. 

 As we discussed in the introduction, we can also view this as the input data for a gauged linear sigma model, with gauge group $T$. Associated to the model there is a complex orbifold $\cF$, which is supposed to be the space of possible values for the complexified Fayet--Iliopoulos parameter that occurs in the Lagrangian of the model. Fortunately we don't need to understand what this means, since toric mirror symmetry provides a precise heuristic recipe for constructing $\cF$ (see for example \cite{CCIT}). We'll now describe this recipe.

Let $\Pi$ be the polytope in $\Z^{n-r}$ spanned by the images of the canonical generators for $\Z^n$. For simplicity, we assume that all these images are distinct, and that $\Pi$ contains only these $n$ lattice points and no others. Now pick $(n-r)$ formal variables $x_1,\ldots,x_{n-r}$, so each point $p=(p_1,\ldots,p_{n-r})\in \Z^{n-r}$ corresponds to a Laurent monomial $x_1^{p_1} \ldots x_{n-r}^{p_{n-r}}$, which we denote by $x^p$. We consider the set of Laurent polynomials
$$ L = \setconds{ \displaystyle \sum^{}_{p\in \Pi} \dualCoord{p} x^p }{ \dualCoord{p} \in \C } $$ 
spanned by Laurent monomials corresponding to the vertices of $\Pi$. Then $L$ is naturally the dual vector space to $V$, and it carries an action of $(\C^*)^{n-r}$ by rescaling each variable $x_1,\ldots,x_{n-r}$. The weight matrix for this $(\C^*)^{n-r}$-action is exactly $P^T$, so if we try and quotient $L$ by this action, then we are exactly looking at the dual GIT problem. 

 The FI parameter space $\cF$ is defined to be the quotient of a certain open set in $L$ by the torus $(\C^*)^{n-r}$. To obtain this open set, we remove two kinds of points from $L$, as follows:
\begin{enumerate}
\item \label{itm:hyperplanes} The hyperplanes $\set{\dualCoord{p}=0}$, for every point $p\in \Z^{n-r}$ that corresponds to a vertex of $\Pi$. In other words, we insist that our Laurent polynomials have Newton polytope exactly $\Pi$. 
\end{enumerate}
We'll call the complement of this set of hyperplanes $L'$. Every point in $L'$ has finite stabilizer, so the space 
$$\overline{\cF} := \quotstack{L'}{(\C^*)^{n-r}}$$
 is an orbifold. It's actually a (non-compact) toric orbifold, via the residual action of the torus $T^\vee$. The space $\cF$ is an open set in $\overline{\cF}$, which we obtain by removing a second kind of point:
\begin{enumerate}
\addtocounter{enumi}{1}
\item \label{itm:discriminantlocus} The \quotes{discriminant} locus of `non-generic' polynomials. Genericity here means the following: for every face $F$ of $\Pi$ (of any dimension) that doesn't contain the origin, consider the `restricted' Laurent polynomial obtained by setting to zero each coefficient that doesn't correspond to a vertex of $F$. We require that every such restricted Laurent polynomial has no critical points in the torus $(\C^*)^{n-r}$.
\end{enumerate}

The discriminant locus is the degeneracy locus of the associated GKZ system of differential equations \cite{Adolphson94}. The discriminant locus is not usually invariant under the action of $T^\vee$, so there is no torus action on $\cF$.
 
The $1$-parameter subgroups in $T^\vee$ correspond to characters $\theta$ for our original GIT problem. They're divided into chambers corresponding to the different possible GIT quotients $X$, these are the chambers of the secondary fan. For each possible GIT quotient we have a corresponding `large-radius limit' in $\cF$, which is roughly the limit of a generic point in $\overline{\cF}$ under the action of any of the 1-parameter subgroups in the corresponding chamber. If this limit exists then it will be a torus fixed point in $\overline{\cF}$, if not then we think of it as a region lying at infinity. 

\begin{rem} Although this recipe appears to correctly produce the FI parameter space, it is not enough to accurately produce the mirror family. By construction, $\cF$ is a moduli space of Laurent polynomials, and Hori--Vafa \cite{HV} argued that the mirror family is the associated family of Landau--Ginzburg models. This works in some compact examples, but is known to be only an approximation when the phases are non-compact, as happens in our examples. The actual construction of the mirrors to our examples is rather subtle, see \cite{ChanPomerleanoUeda}.
\end{rem}

\subsection{Results for our examples}
\label{sect.resultsForExamples}
Now we run the recipe of the previous section on the toric varieties that we're interested in. Fix a partition $\Gamma$ encoded by
$$\gamma:[0,k] \to [0,s], $$
and let $X_\Gamma$ be the corresponding toric variety. In Section~\ref{sect.subvarieties} we computed the full toric fan for $X_\Gamma$, in particular we described (Lemma~\ref{lem.raygenerators}) the set of generators for all the rays in the fan.  This is the same as giving the matrix $P$, so it's enough information to run the recipe and compute the corresponding FI parameter space $\cF_\Gamma$.

 Recall that the rays in the fan for $X_\Gamma$ are generated by the vectors
$$\polyVec{t}{\delta} \in \Z^{2 + (s+1)}$$
defined in \eqref{eq.polyVecDef}. Here  $t\in [0,s]$ and $\delta\in [0, N_t]$, where $N_t =\#\gamma^{-1}(t)$. The corresponding Laurent monomials are
$$\tilde{u}^\delta\tilde{v}^{N_t-\delta}\tilde{z}_t, $$
so the space $L$ consists of Laurent polynomials of the form
$$f_0(\tilde{u},\tilde{v})\tilde{z}_0 + \ldots + f_s(\tilde{u},\tilde{v})\tilde{z}_s  $$
where each $f_t$ is a homogeneous polynomial of degree $N_t$. We have a $(\C^*)^{2 +(s+1)}$ action on $L$, but there is a global $\C^*$ stabilizer which is an artifact of our decision to draw the toric fan in a lattice whose rank was 1 larger than necessary. We need to remember to neglect this global stabilizer when forming the quotient.
 
Next we have to identify which loci we should delete from $L$. Step~\ref{itm:hyperplanes} is easy, we just require that the first and last coefficients of each $f_t$ cannot go to zero, i.e. each polynomial $f_t(\tilde{u},1)$ cuts out a length $N_t$ subscheme in the punctured line $\C^*_{\tilde{u}}$.

 For step~\ref{itm:discriminantlocus}, recall \eqref{eq.prism2} that the vectors $\polyVec{t}{\delta}$ span a polytope $\tilde{\Pi}$, which is abstractly isomorphic to a prism $\Delta^s\times I$. Hence there are two kinds of faces of $\tilde{\Pi}$. For any subset $J\subset [0, s]$, there are subsimplices $\Delta_J\subset \Delta^s$ at either end of the prism, and there is also the prism $\Delta_{J}\times I \subset  \Delta^s\times I$.
 
 \begin{itemize}
 \item For the first kind of face, the corresponding restricted Laurent polynomial is of the form
$$ \sum_{t \in J} \lambda_t \tilde{u}^{N_t} \tilde{z}_t \hspace{1cm}\mbox{or}\hspace{1cm}  \sum_{t \in J} \mu_t \tilde{v}^{N_t} \tilde{z}_t, $$ 
and these never have critical points on the torus.
\item For the second kind of face, the corresponding restricted Laurent polynomial is
$$ \sum_{t \in J} f_t(\tilde{u},\tilde{v}) \tilde{z}_t.  $$
Now we have a non-trivial condition, because this Laurent polynomial has a critical point in the torus if and only if the corresponding complete intersection 
$$\set{f_t(\tilde{u},1) = 0, \; \forall t\in J} \subset \C^*_{\tilde{u}} $$
is non-generic. When $\#J =1$, this gives us the condition that no single $f_t$ can have a repeated root. When $\#J\geq 2$, we get the condition that no two of the $f_t$ can have a shared root.
\end{itemize}

In summary, we have the following proposition.
\begin{prop}$$\cF_\Gamma = \set{ (f_0,\ldots,f_s) } \;/\; (\C^*)^{s+2} $$
where each $f_t$ is a homogeneous polynomial of degree $N_t$ in $\tilde{u}$ and $\tilde{v}$ such that
\begin{itemize}
\item no $f_t$ has roots at $\tilde{u}=0$ or $\tilde{v}=0$,
\item no $f_t$ has repeated roots, and
\item no two of the $f_t$ share a root.
\end{itemize}
The torus acts by rescaling $\tilde{u}$ and $\tilde{v}$, and each $f_t$.
\end{prop}

Now consider the special case when we choose the finest possible partition $\Gamma_{\fin}$, so that each $N_t=1$. In this case we can replace each linear function $f_t$ with its root $\zeta_t\in \C^*_{\tilde{u}}$, so we have (using the abbreviation $\cF_{\fin}:=\cF_{\Gamma_{\fin}}$)
$$\cF_{\fin} = \set{(\zeta_0\!:\!\ldots\!:\!\zeta_k)} \;\;\subset\;\; \P^k $$
where the $\zeta_t$ are all distinct and non-zero.

At the opposite extreme, suppose we choose the coarsest possible partition $\Gamma_{\crs}$. This has only one piece, so $s=0$ and $N_0=k+1$. In this case, we can replace the single $f_0$ with its set of roots $\set{\zeta_0,\ldots,\zeta_k}\subset \C^*_{\tilde{u}}$. Writing $\cF_{\crs}:=\cF_{\Gamma_{\crs}}$, we then have 
$$\cF_{\crs} = \quotstackBig{\set{(\zeta_0\!:\!\ldots\!:\!\zeta_k)}}{S_{k+1}}  \;\;\subset\;\; \quotstackBig{\P^k}{S_{k+1}} $$
where again the $\zeta_t$ are all distinct and non-zero. So $\cF_{\crs}$ is a quotient of $\cF_{\fin}$  by the action of the symmetric group.

 Finally, suppose that $\Gamma$ is some intermediate partition. We can identify each $f_t$ with its set of roots, and so get a set of distinct non-zero roots $(\zeta_0,\ldots,\zeta_k)$. However, this set of roots is partitioned by $\Gamma$, and we allow relabelings that preserve the partition. In other words, we have
$$\cF_\Gamma =\quotstack{\cF_{\fin}}{S_\Gamma}$$
where $S_\Gamma$ is the Young subgroup
$S_\Gamma = S_{N_0}\times \ldots\times S_{N_s} \subset S_{k+1} $
that fixes the partition $\Gamma$.

\subsection{Large-radius limits}
\label{sect.LRLs}
We now identify some of the large-radius limits in the spaces $\cF_\Gamma$. Pick a 1-parameter subgroup of $T^\vee$ corresponding to a character $\theta \in (\Z^{k+1})^\vee$ with a lift $\vartheta\in (\Z^{2(k+1)})^\vee$  lying in the standard chamber \eqref{standardtheta}. This acts on the space 
$\cF_{\fin}$ by 
$$\theta(\lambda): \zeta_i \mapsto \lambda^{\vartheta_i} \zeta_i, $$
so the corresponding `region at infinity' in $\cF_{\fin}$ is the locus where
\begin{equation}\label{standardLRlimit}
\log|\zeta_0|\gg\log |\zeta_1|\gg \ldots\gg \log|\zeta_k|.
\end{equation}
The other large-radius limits in $\cF_{\fin}$ are of the same form, but with the order of the roots permuted by some $\sigma \in S_{k+1}$.

Next we look at $\cF_{\crs}$, where we have a single polynomial $f_0$. A 1-parameter subgroup $\tilde{\theta}$ in the standard chamber is induced from a standard $\theta$ as explained in Section~\ref{sect.subvarietyGITdescrip}, and it acts on the polynomial $f_0$ by rescaling the coefficient of $\tilde{u}^{k+1-\delta}\tilde{v}^{\delta}\tilde{z}_0$ with weight 
$$ \tilde{\vartheta}(\polyVec{0}{\delta}) = \vartheta_0+\vartheta_1 +\ldots+ \vartheta_{\delta-1}$$
(see \eqref{eqn.tildevartheta}). Therefore it acts on the roots of $f_0$ by rescaling them with weights $\vartheta_0,\ldots, \vartheta_k$, and so the standard LR limit in $\cF_{\crs}$ is again the region where
$$\log|\zeta_0|\gg\log |\zeta_1|\gg \ldots\gg \log |\zeta_k| $$
(here the labelling of the roots is arbitrary, so we can choose to label them according to size). So the standard LR limit in $\cF_{\crs}$ is the common image of all the LR limits in $\cF_{\fin}$, exactly as our heuristic picture suggests.

Now leave the standard chamber by crossing the wall $\tilde{\theta}(\tau_i) = \vartheta_{i-1}-\vartheta_i =0$, i.e. violate the inequality
$$ 2 \tilde{\vartheta}(\polyVec{0}{i}) - \tilde{\vartheta}(\polyVec{0}{i+1}) - \tilde{\vartheta}(\polyVec{0}{i-1}) >0 $$
while preserving all the other such inequalities. The corresponding LR limit in $\cF_{\crs}$ is the region where
$$ \log |\zeta_0|\gg  \ldots\gg \log |\zeta_{i-1}| \approx\log  |\zeta_{i}|\gg \ldots\gg\log  |\zeta_k| $$
and
$$ \log |\zeta_{i-1} + \zeta_i | \ll \log |\zeta_{i}|.$$ 
These are the LR limits which are `adjacent' to the standard one, if we go further into the parameter space (i.e. cross further walls in the secondary fan) we reach other limits where more roots become commensurable.

 Now let $\Gamma$ be an arbitrary partition, and $\cF_\Gamma$ the corresponding parameter space. From the discussion above, the standard LR limit in $\cF_\Gamma$ is the image of the standard LR limit in $\cF_{\fin}$ (and all LR limits obtained from the standard one under the action of $S_\Gamma$), as the heuristic picture suggests. When we leave the standard chamber by crossing the wall $\tilde{\theta}(\tau_i)=0$, there are two possibilities:
\begin{list}{(\alph{enumi})}{\usecounter{enumi}}
\item If $i-1$ and $i$ are in distinct pieces of $\Gamma$, then we move to a LR limit in $\cF_\Gamma$ which is the image of a LR limit in $\cF_{\fin}$. It's obtained from the standard limit by transposing $i-1$ and $i$.
\item If $i-1$ and $i$ are in the same piece of $\Gamma$, then we move to a LR limit in $\cF_\Gamma$ where the two roots $\zeta_{i-1}$ and $\zeta_i$ of the polynomial $f_{\gamma(i)}$ have comparable sizes.
\end{list}

The phase corresponding to this new LR limit was discussed in Section~\ref{sect.familiesofrationalcurves}.

\subsection{Fundamental groups}
\label{sect.fundamentalgroups}
In this section we make some simple observations on the fundamental groups of the FI parameter spaces computed in the previous section. These spaces are orbifolds, and the symbol $\pi_1$ will always denote the orbifold fundamental group.

  Consider the FI parameter space $\cF_{\crs}$ associated to the surface $X_{\Gamma_{\crs}}$. As was noted in \cite{CCIT}, it has fundamental group
$$\pi_1(\cF_{\crs}) = \tilde{B}_{k+1} \rtimes C_{k+1}, $$
where $\tilde{B}_{k+1}$ is the affine braid group associated to the affine Dynkin diagram $\tilde{A}_k$, and $C_{k+1}$ is a cyclic group. We should explain this briefly: by rescaling we can insist that $\zeta_0\zeta_1\ldots\zeta_k = 1$, and this leaves a residual action of the cyclic group $C_{k+1}$. So if 
$$\mathcal{G} = \setconds{\zeta_0,\ldots,\zeta_k}{ \prod_i \zeta_i = 1 } $$
with all the $\zeta_i$ distinct, then $\cF_{\crs}= [\mathcal{G}\,/\,S_{k+1}\times C_{k+1}]$. By taking $\log$s of all the $\zeta_i$ we get a principle $\Z^k$-bundle over $\mathcal{G}$ given by
$$\overline{\mathcal{G}} = \setconds{ w_0,\ldots,w_k }{ \sum_i w_i=0,\, w_i-w_j\notin \Z \mbox{ for } i\neq j}, $$
and this is the complement of the affine complex hyperplane arrangement associated to $\tilde{A}_k$. When we quotient by the affine Coxeter group $\Z^k\rtimes S_{k+1}$, we get
$$\pi_1(\mathcal{G}\,/\,S_{k+1}) = \tilde{B}_{k+1},$$
and so $\pi_1(\cF_{\crs})$ is $\tilde{B}_{k+1} \rtimes C_{k+1}$, as claimed.

 We want to take a slightly different point of view on this group. Choose a point in $\cF_{\crs}$ which is the orbit of a tuple $(\zeta_0,\zeta_1,\ldots,\zeta_k)$ with
$$\log  |\zeta_0|  \gg \ldots \gg \log |\zeta_k|.  $$
We can think of such a point as lying `near the large-radius limit'  for the standard phase. If we stay in this region, we can only see a subgroup of $\pi_1(\cF_{\crs})$ which is a lattice $\Z^k = \Z^{k+1}/\Z$ coming from rotating the phases of the $\zeta_i$. 
 
Alternatively, we could leave this large-radius limit, but insist that we stay in (the orbit of) the region where $\Re(\zeta_i)<0$ for all $i$. The subgroup of $\pi_1(\cF_{\crs})$ contributed by this region is the ordinary braid group $B_{k+1}$. These two subgroups generate $\pi_1(\cF_{\crs})$, indeed we can view it as a semidirect product
$$\pi_1(\cF_{\crs}) = \langle \Z^k \rangle \rtimes B_{k+1},$$
where $\langle \Z^k \rangle$ is the normal closure of $\Z^k$.

 From this, it's easy to deduce $\pi_1(\cF_\Gamma)$ for any other partition $\Gamma$. The inclusion of $\cF_{\crs}$ into $[\P^k / S_{k+1}]$ gives a map from $\pi_1(\cF_{\crs})$ to $S_{k+1}$, and $\pi_1(\cF_\Gamma)$ is the fibre product of  $\pi_1(\cF_{\crs})$ with the Young subgroup $S_\Gamma$. So it's generated by a lattice $\Z^k$, and the mixed braid group $B_\Gamma$, which by definition \eqref{eqn.mixed_braid_defn} is the fibre product
$$B_\Gamma:= B_{k+1}\times_{S_{k+1}} S_\Gamma.$$

\section{Mixed braid group actions}
\label{sect.mixedbraidgroupactions}

For each $\Gamma$, we wish to produce an action of $\pi_1(\cF_\Gamma)$ on the derived category $\Db(X_\Gamma)$. The FI parameter space heuristics give a precise prediction for what this action should be, as we now explain.

\subsection{Heuristics for the generators}\label{sect.heuristicsforthegenerators}

Pick a partition $\Gamma$, and pick a base point in $\cF_\Gamma$ which is close to the large-radius limit for the standard phase, i.e. it is the orbit of a tuple $(\zeta_0,\ldots,\zeta_k)$ where
$$\log |\zeta_0|\gg \ldots\gg \log |\zeta_k|. $$
Recall from Section~\ref{sect.fundamentalgroups} that $\pi_1(\cF_\Gamma)$ is generated by two subgroups: a lattice $\Z^k$ arising from rotating the phases of the $\zeta_i$, and the mixed braid group $B_\Gamma$.  The action of the lattice on $\Db(X_\Gamma)$ is obvious,  we can canonically identify this lattice with the set of toric line bundles on $X_\Gamma$, and these act on the derived category by the tensor product. Consequently we will ignore the lattice in all of the subsequent discussion, and focus on the mixed braid group $B_\Gamma$. This means we only consider the region in $\cF_\Gamma$ which is the orbit of the set $$\{\Re(\zeta_i)< 0, \forall i\}.$$  Let's denote this region by $\cF_\Gamma^{<0}$.

  If we let $\Gamma$ vary over all partitions (of size $k+1$), then the spaces $\cF_\Gamma^{<0}$ form a poset of covering spaces, with $\cF_{\fin}^{<0}$ at the top and $\cF_{\crs}^{<0}$ at the bottom. Let's apply some completely elementary algebraic topology to this situation. Pick a base point $b\in \cF_{\crs}$ which is close to the standard LR limit. For any $\Gamma$, the preimage of $b$ in $\cF_\Gamma^{<0}$ is some finite set of points, which we can identify with the set of cosets 
$$ \pi_1(\cF_{\crs}^{<0}) \,/\, \pi_1(\cF_\Gamma^{<0}) = B_{k+1}\,/\, B_\Gamma = S_{k+1}\,/\,S_\Gamma $$
by matching the identity coset with the point near the standard LR limit in $\cF_\Gamma^{<0}$.

Let $\cG_\Gamma$ be the groupoid of homotopy classes of paths in $\cF_\Gamma^{<0}$ that start and end somewhere in this set of points. This is the same thing as the action groupoid for the action of $B_{k+1}$ on this set of cosets. The isotropy group in $\cG_\Gamma$ is of course $B_\Gamma$. If we let $\Gamma$ vary we get a set of groupoids, and the covering maps induce functors between them which are faithful, and also surjective on the total sets of arrows. 

Every point $\sigma S_\Gamma$ in $\cG_\Gamma$ lies near a LR limit in $\cF_\Gamma$, and there is a corresponding phase $X_\Gamma^\sigma$ of the GIT problem. If we move along a path in the groupoid to another point $\sigma' S_\Gamma$, the FI parameter space heuristics predict that we get an associated derived equivalence
$$\Db(X_\Gamma^\sigma) \isoto \Db(X_\Gamma^{\sigma'}).$$
Putting these together, we expect to get a functor
$$T: \cG_\Gamma \to \Catni $$
sending $\sigma S_\Gamma$ to $\Db(X_\Gamma^\sigma)$. Here $\Catni$ denotes the category of categories with morphisms being functors up to natural isomorphism.\footnote{This is the natural 1-category associated to the 2-category of categories.}

 The FI parameter space heuristics can also tell us exactly what $T$ should do to arrows. To explain this, we need to identify some generators for $G_\Gamma$. 

Firstly set $\Gamma=\Gamma_{\fin}$. The resulting groupoid $\cG_{\fin}$ is very easy to describe: its arrows are ordinary braid diagrams but with the strands labelled from $0$ to $k$ in some order. Pick a point $\sigma\in S_{k+1}$ in this groupoid, and suppose that $i$ and $j$ are adjacent labels at this point (i.e. $i$ and $j$ are adjacent after applying $\sigma$ to $[0,k]$), with $i$ to the left of $j$. Let $t^\sigma_{i, j}$ be the braid which crosses these two adjacent strands (left-over-right, as in Figure~\ref{fig.mixedbraid}) and leaves all the other strands alone. There is a corresponding path in $\cF_{\fin}$, starting in the LR limit
\begin{equation}\label{eqn.LRlimitsigma} \log |\zeta_{\sigma(0)}|\gg \ldots\gg\log |\zeta_i|\gg\log|\zeta_j|\gg\ldots \gg\log|\zeta_{\sigma(k)}| \end{equation}
and ending in the LR limit
$$ \log |\zeta_{\sigma(0)}|\gg \ldots\gg\log |\zeta_j|\gg\log|\zeta_i|\gg\ldots \gg\log|\zeta_{\sigma(k)}| $$
where only $\zeta_i$ and $\zeta_j$ change in value along the path, their norms change monotonically, and we have $\arg(\zeta_i)>\arg(\zeta_j)$ at the point where their norms are equal. 

 The set of all these braids (or paths) $t^\sigma_{i,j}$, for all $\sigma$, evidently generates the groupoid $\cG_{\fin}$. Therefore for any partition $\Gamma$, their images under the map
$$\cG_{\fin}\to \cG_\Gamma$$
generate $\cG_\Gamma$, since the map is surjective on the total set of arrows.  We'll denote the image of $t^\sigma_{i,j}$ by $t^{\sigma S_\Gamma}_{i,j}$, since it only depends on the coset $\sigma S_\Gamma$.  When $\Gamma = \Gamma_{\crs}$ we just have the standard generators for $\cG_{\crs}=B_{k+1}$, indeed  the set $t^{\sigma S_\Gamma}_{i,j}$ are just the preimages of these standard generators.

 Now fix an arbitrary $\Gamma$ again, and let 
$$t^{\sigma S_\Gamma}_{i,j}: \sigma S_\Gamma  \to \sigma' S_\Gamma $$
be one of these generating arrows, where $\sigma'$ is the composition $\sigma\cdot(ij)$. Let's describe the corresponding path in $\cF_\Gamma^{<0}$, and deduce the derived equivalence that it should correspond to. There are two cases to consider:

\begin{list}{(\alph{enumi})}{\usecounter{enumi}}
\item \label{casea} If $i$ and $j$ don't lie in the same part of $\Gamma$ then this path moves us from one LR limit to a neighbouring LR limit, and the two phases $X_\Gamma^\sigma$ and $X_\Gamma^{\sigma'}$ are related by a flop of a trivial family of rational $(-1,-1)$-curves, as described in Section~\ref{sect.familiesofrationalcurves}. The path $t^{\sigma S_\Gamma}_{i,j}$ should correspond to the derived equivalence
$$T^{\sigma S_\Gamma}_{i,j} := \psi_0 : \Db(X_\Gamma^\sigma) \isoto \Db(X_\Gamma^{\sigma'}) $$
discussed in Section~\ref{sect.derivedequivalencesfromwalls}, which as we saw can be defined either using VGIT and `windows', or geometrically using the birational roof.

\item \label{caseb} If $i$ and $j$ lie in the same part of $\Gamma$ then $\sigma S_\Gamma$ and $\sigma' S_\Gamma$ are the same coset, and the path $t^{\sigma S_\Gamma}_{i,j}$ is actually a loop. We
start at the LR limit \eqref{eqn.LRlimitsigma}, move into the neighbouring LR limit where $|\zeta_i|\approx |\zeta_j|$, and return, having looped the discriminant locus $\zeta_i=\zeta_j$. The phase $X_\Gamma^\sigma$ contains a trivial family of rational $(-2)$-curves. When we pass to the neighbouring LR limit this family gets `flopped' and becomes a family of $\Z_2$-orbifold points. Consequently this loop in $\cF_\Gamma^{<0}$ should correspond to the autoequivalence
$$ T^{\sigma S_\Gamma}_{i,j}:= (\psi_{-1})^{-1}\circ \psi_0: \Db(X_\Gamma^\sigma) \isoto \Db(X_\Gamma^{\sigma}) $$
discussed in Section~\ref{sect.derivedequivalencesfromwalls}, which is equal to a family spherical twist around the family of rational curves.

\end{list}

If we let $\cG_\Gamma^{\free}$ be the free groupoid  generated by the arrows $t^{\sigma S_\Gamma}_{i,j}$, then it's tautological that this assignment
$$ T_\Gamma: t^{\sigma S_\Gamma}_{i,j} \to T^{\sigma S_\Gamma}_{i,j} $$
defines a functor from $\cG_\Gamma^{\free}$ to $\Catni$. In agreement with the FI parameter space heuristics, we will show that in fact we have:

\begin{thm}\label{thm.mainthm} $T_\Gamma$ gives a well-defined functor
$$T_\Gamma: \cG_\Gamma \to  \Catni.$$
 \end{thm}
In the process we will show the following stronger result, for which we don't have a heuristic justification:
\begin{thm}\label{thm.faithfulness} $T_\Gamma$ is faithful.
\end{thm}
A priori this second theorem is much more difficult. However for the special case when $\Gamma=\Gamma_{\crs}$, and so $X_\Gamma$ is a surface, the faithfulness was proved by Seidel and Thomas \cite{ST}, and we can deduce the result for general $\Gamma$ fairly easily.

Restricting $T_\Gamma$ to the isotropy group of the standard phase we deduce:

\begin{cor}\label{cor.maincor} $T_\Gamma$ defines a faithful action of the mixed braid group $B_\Gamma$ on the derived category $\Db(X_\Gamma)$.\end{cor}

The proofs of these two theorems will be presented in the next two sections.

\subsection{Braid relations on the ambient space}
\label{sect.braid_relns_ambient_space}

In this section we'll prove the following special case of Theorem~\ref{thm.mainthm}:

\begin{prop}\label{prop.braidrelationsonversal} When $\Gamma=\Gamma_{\fin}$, so each phase $X_\Gamma^\sigma$ is the versal deformation space of the resolution of the $A_k$ singularity, then $T_{\fin} := T_{\Gamma_{\fin}}$  gives a well-defined functor
$$ T_{\fin} : \cG_{\fin} \to \Catni. $$
\end{prop}

The key case to consider is when $k=2$. (When $k=1$, there's nothing to prove.) In this case there are six phases $X^\sigma$, indexed by elements $\sigma$ of $S_3$. We can denote the functors between them by $T^\sigma_{i,j}$, dropping the trivial group $S_{\Gamma_{\fin}}$ from the notation. Also, we'll write elements of $S_3$ as orderings of the set $\{0,1,2\}$ rather than as products of cycles, so for example the standard phase will be denoted $X^{012}$. 

A prototypical braid relation in $\cG_{\fin}$ is drawn in Figure~\ref{fig.braidrelation}, and expressed in the following diagram of functors:

\begin{prop}\label{prop.hexagoncommutes}
The hexagon below commutes.
\begin{equation}\label{eqn.hexagon}
\begin{tikzcd}
 \, & \Db(X^{102}) \arrow{r}{T_{0,2}^{102}}   &   \Db(X^{120}) \arrow{dr}{T^{120}_{1,2}} &  \, \\
 \Db(X^{012}) \arrow{ur}{T^{012}_{0,1}} \arrow{dr}[swap]{T^{012}_{1,2}}& \, &  \,  &  \Db(X^{210})  \\
 \,&  \Db(X^{021}) \rar[swap]{T^{021}_{0,2}}   &   \Db(X^{201}) \urar[swap]{T^{201}_{0,1}} &  \,
\end{tikzcd}
 \end{equation}
\end{prop}

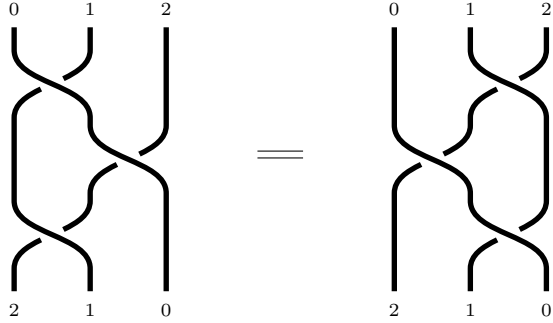
\begin{figure}[h]
\begin{center}
\begin{tikzpicture}
	\braid[number of strands=3,line width=2pt] (B1) s_1 s_2 s_1;
	\node[at=(B1-1-s), yshift=7pt] {\scriptsize 0};
          \node[at=(B1-2-s),  yshift=7pt] {\scriptsize 1};
	\node[at=(B1-3-s),  yshift=7pt] {\scriptsize 2};
	\node[at=(B1-1-e), yshift=-7pt] {\scriptsize 0};
          \node[at=(B1-2-e),  yshift=-7pt] {\scriptsize 1};
	\node[at=(B1-3-e),  yshift=-7pt] {\scriptsize 2};
\draw (4.2,-48+1.5*\sepForEquals) -- (4.8,-48+1.5*\sepForEquals);
\draw (4.2,-48-1.5*\sepForEquals) -- (4.8,-48-1.5*\sepForEquals);
	\braid[number of strands=3,line width=2pt] (B2) at (6,0) s_2 s_1 s_2;
	\node[at=(B2-1-s), yshift=7pt] {\scriptsize 0};
          \node[at=(B2-2-s),  yshift=7pt] {\scriptsize 1};
	\node[at=(B2-3-s),  yshift=7pt] {\scriptsize 2};
	\node[at=(B2-1-e), yshift=-7pt] {\scriptsize 0};
          \node[at=(B2-2-e),  yshift=-7pt] {\scriptsize 1};
	\node[at=(B2-3-e),  yshift=-7pt] {\scriptsize 2};
\end{tikzpicture}
\end{center}
\caption{Prototypical relation in $\cG_{\fin}$ for $k=2$.}
\label{fig.braidrelation}
\end{figure}

This is a special case of Proposition~\ref{prop.braidrelationsonversal}, but it quickly implies the whole proposition:

\begin{proof}[Proof of Proposition~\ref{prop.braidrelationsonversal}] Let $k>2$. We can draw a similar hexagon moving between the standard phase $X^{0123\ldots k}$ on the left, and the phase $X^{2103\ldots k}$ on the right, and the commutativity of this hexagon is one of the braid relations. If we start with the fan for the standard phase, as described in Section~\ref{sect.toriccalculations}, then the changes to the fan that take place when we move around the hexagon will only affect the first three cones: the remaining $\Delta_3, \ldots, \Delta_k$ stay unchanged, and so the birational modifications are concentrated inside the open set $\set{a_3\neq 0,\ldots, a_k\neq 0}$. Inside this open set, the geometry we're studying is just a trivial family (over $\C^{k-2}_{b_3,\ldots,b_k}$) of copies of the geometry appearing in the $k=2$ case.  Therefore Proposition~\ref{prop.hexagoncommutes} implies that this hexagon also commutes. 

We've shown that one particular braid relation holds for each $k\geq 2$, so all the others hold by the  $S_{k+1}$ symmetry. 
\end{proof}

The rest of this section will be devoted to the proof of Proposition~\ref{prop.hexagoncommutes}, our prototype braid relation.

\subsubsection{Formal structure of the proof}
\label{sect.formalStructureOfProof}
To see why the hexagon \eqref{eqn.hexagon} commutes, we need to go back to the fundamental definition of our functors $T^{\sigma}_{i,j}$ from Section~\ref{sect.derivedequivalencesfromwalls}.

 Let  $V\cong \C^6$ and $T \cong (\C^*)^3$ be the data of the GIT problem that we're considering, and let $T' = T / \C^*$ be the quotient of $T$ by the trivially-acting diagonal subgroup.  Let $\cX$ denote the Artin stack
$$ \cX = \quotstack{ V }{ T' }. $$
Each of the six phases $X^\sigma$ of the GIT problem is an open substack in $\cX$. Let's identify these open substacks completely explicitly. 

Let $\tau_0, \tau_1$ and $\tau_2$ be the three $1$-parameter subgroups of $T$ corresponding to the nodes of the quiver, so in the quotient torus $T'$ we have $\tau_0 + \tau_1 + \tau_2=0$. These three are the only $1$-parameter subgroups which fix more than the origin in $V$, and their fixed loci are as follows:
\begin{align*} 
\tau_0 = (1,0,0) \;\mbox{ fixes }\; \bar{Z_0} = \set{a_0=b_0=a_2=b_2=0}, \\
\tau_1 = (0,1,0) \;\mbox{  fixes }\; \bar{Z_1} = \set{a_0=b_0=a_1=b_1=0}, \\
\tau_2 = (0,0,1) \;\mbox{  fixes }\; \bar{Z_2} = \set{a_1=b_1=a_2=b_2=0}.
\end{align*}
If we pick a stability condition, then each $\tau_i$ destabilizes either the attracting or the repelling subspace of the corresponding $\bar{Z_i}$, and the union of these three subspaces is precisely the unstable locus for that stability condition. For example, in the standard phase $X^{012}$ the unstable locus is the union of the three subspaces
$$\bar{S_0} = \set{b_0=a_2=0}, \hspace{1cm}
\bar{S_1} = \set{b_0=a_1=0}, \hspace{1cm}
\bar{S_2} = \set{b_1=a_2=0}. $$
(see also Figure~\ref{fig.strata}). As another example, if we move to the adjacent phase $X^{021}$ then only $\bar{S_2}$ changes: it gets replaced by the subspace $\bar{S'_2} = \set{a_1=b_2=0}$.

\begin{figure}[h]
\begin{center}
\begin{tikzpicture}[>=stealth,scale=1.5,node distance=10]
\quiver{->}{->}{dotarrow}{->}{->}{dotarrow}{0}{0}{S0} ;
\quiver{->}{dotarrow}{dotarrow}{->}{->}{->}{0}{-1.5}{S1} ;
\quiver{dotarrow}{->}{->}{->}{->}{dotarrow}{0}{-3}{S2} ;

\quiver{->}{->}{dotarrow}{dotarrow}{dotarrow}{dotarrow}{5}{0}{Z0} ;
\quiver{dotarrow}{dotarrow}{dotarrow}{dotarrow}{->}{->}{5}{-1.5}{Z1} ;
\quiver{dotarrow}{dotarrow}{->}{->}{dotarrow}{dotarrow}{5}{-3}{Z2} ;

% labels
\node[right of=S0Centre, node distance=60] (S0Label) {$\bar{S}_0$};
\node[right of=S1Centre, node distance=60] (S1Label) {$\bar{S}_1$};
\node[right of=S2Centre, node distance=60] (S2Label) {$\bar{S}_2$};

\node[left of=Z0Centre, node distance=60] (Z0Label) {$\bar{Z}_0$};
\node[left of=Z1Centre, node distance=60] (Z1Label) {$\bar{Z}_1$};
\node[left of=Z2Centre, node distance=60] (Z2Label) {$\bar{Z}_2$};

\draw[->] (S0Label) -- (Z0Label) node[above,pos=0.5]{\scriptsize $\tau_0$};
\draw[->] (S1Label) -- (Z1Label) node[above,pos=0.5]{\scriptsize $\tau_1$};
\draw[->] (S2Label) -- (Z2Label) node[above,pos=0.5]{\scriptsize $\tau_2$};

\end{tikzpicture}
\end{center}
\caption{For the standard phase $X^{012}$, unstable subspaces $\bar{S_i}$ flowing under $1$-parameter subgroups $\tau_i$ to fixed loci $\bar{Z_i}$: each locus is defined by setting the dotted arrows to zero.}
\label{fig.strata}
\end{figure}
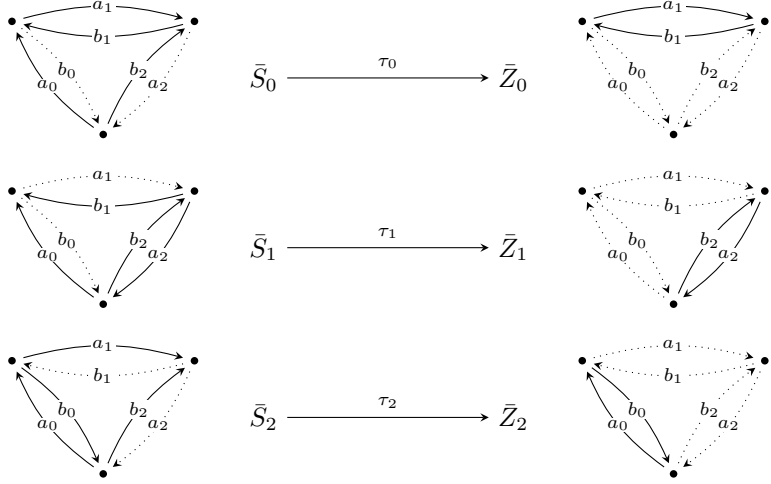

Now  pick a (non-zero) stability condition that lies on the wall between the two phases $X^{012}$ and $X^{021}$, so the corresponding unstable locus is $\bar{S_0}\cup \bar{S_1}$. The partial quotient
$$Y^{012}_{1,2} := (V - \bar{S_0}- \bar{S_1}) \,/\,\tau_0 $$
is smooth, and carries a residual action of $\tau_2$. The stack
$$\cY^{012}_{1,2} = \quotstack{(V - \bar{S_0}- \bar{S_1})}{T'} =  \quotstack{Y^{012}_{1,2}}{\tau_2} $$
is an open substack of $\cX$, and it contains both $X^{012}$ and $X^{021}$ as open substacks. We can view this as a GIT problem, and it describes this particular wall-crossing.
 
 Recall that we defined windows
$$\cW^{012}_{1,2}(k) \subset \Db(\cY^{012}_{1,2}) $$
as the full subcategories
$$\cW^{012}_{1,2}(k) = \genconds{\cE}{\begin{aligned} &\mbox{all homology sheaves of }L\sigma^* \cE\mbox{ have $\tau_2$-weights}\\[-5pt] &\mbox{lying in the interval } [k, k + 2) \end{aligned}} $$
where $\sigma: \bar{Z_2} \into Y^{012}_{1,2}$ is the inclusion of the fixed locus (recall also that we've calculated that the numerical invariant $\eta$ is 2). Each of these windows is equivalent, under the restriction functor, to the derived categories of both $X^{012}$ and $X^{021}$, and the wall-crossing functor $T^{012}_{1,2}$ is defined by the following commuting triangle:
$$\begin{tikzcd}
\, & \cW^{012}_{1,2}(0) \arrow{dl}\arrow{dr} &\, \\
 \Db(X^{012})\arrow{rr}{T^{012}_{1,2}} & \, & \Db(X^{021}) \end{tikzcd} $$
Very similar constructions apply for the other five wall crossings in the hexagon \eqref{eqn.hexagon}.

\begin{rem}  
For the horizontal arrows in \eqref{eqn.hexagon}, the 1-parameter subgroup controlling the wall-crossing is $\tau_0$. However, the value of the stability condition $\theta(\tau_0)$ is changing from negative to positive when we cross these walls, not vice versa. This means that when we define the two corresponding windows, $\cW^{102}_{0,2}(0)$ and $\cW^{021}_{0,2}(0)$, we must measure weights with respect to the subgroup $-\tau_0$, not $\tau_0$.\footnote{The functors $T^{102}_{0,2}$ and $T^{021}_{0,2}$ correspond to left-over-right crossings. If we were to use $\tau_0$-weights then we'd get the functors corresponding to right-over-left crossings.}
\end{rem}

\begin{lem}\label{lem.Vexists} There is a subcategory $\cV\subset \Db(\cX)$ such that, for any wall-crossing in the hexagon \eqref{eqn.hexagon}, the restriction functor
$$ \cV \to \Db(\cY^\sigma_{i,j}) $$
is an embedding with image $\cW^\sigma_{i,j}(0)$.
\end{lem}

We'll discuss the definition of this category $\cV$ shortly. Before we do, let's show that Proposition~\ref{prop.hexagoncommutes} follows from it as an immediate formality.

\begin{proof}[Proof of Proposition~\ref{prop.hexagoncommutes}]
For each wall-crossing we have a commuting diagram
$$\begin{tikzcd}
\, & \cV\arrow{d} \arrow{ddl}\arrow{ddr} &\, \\
\, & \cW^\sigma_{i,j}(0) \arrow{dl}\arrow{dr} &\, \\
 \Db(X^\sigma)\arrow{rr}{T^\sigma_{i,j}} & \, & \Db(X^{\sigma\cdot(ij)}) \end{tikzcd} $$
in which all arrows are equivalences. So we can complete the hexagon \eqref{eqn.hexagon} to the following diagram
$$\begin{tikzcd} 
\,& \Db(X^{102}) \arrow{rr}{T^{102}_{0,2}}   &\, & \Db(X^{120}) \drar{T^{120}_{1,2}} &\, \\
 \Db(X^{012}) \urar{T^{012}_{0,1}} \drar[swap]{T^{012}_{1,2}}&\, &
\cV \arrow{ll} \arrow{ul} \arrow{ur} \arrow{rr} \arrow{dr} \arrow{dl}
  & \,&  \Db(X^{210}) \\
\,&  \Db(X^{021}) \arrow{rr}[swap]{T^{021}_{0,2}}   &\, &   \Db(X^{201}) \urar[swap]{T^{201}_{0,1}} &\,
\end{tikzcd}$$
in which every arrow is an equivalence and every triangle commutes.
\end{proof}

\subsubsection{The category $\cV$}

If we consider a GIT problem where the group acting is just $\C^*$, then as we've explained in Section~\ref{sect.generaltheoryforwindows}, the machinery of \cite{BFK, DHL} tells us how to produce a section of the restriction functor
$$\Db\left(\quotstack{Y^\sigma_{i,j}}{\C^*}\right) \to \Db( Y^\sigma_{i,j} \sslash \C^*), $$
i.e. a lift of the derived category of the GIT quotient into the equivariant derived category of $Y^\sigma_{i,j}$. However, this general machinery doesn't just apply when the group is $\C^*$, it works for any reductive group. In particular, we can apply it to the GIT problem considered in the previous section, and for any phase $X^\sigma$ it will tell us  how to produce a section of the restriction functor
$$\Db(\cX) \to \Db(X^\sigma).$$
We'll now explain how to run this machinery for our example.\footnote{Again, the fact that our group is abelian makes our application considerably simpler than the general case.} We shall see that it constructs a subcategory of $\Db(\cX)$ which (after a little more work) we can show has the right properties to be the category $\cV$ required by Lemma \ref{lem.Vexists}.

 We begin with the standard phase $X^{012}$. We need to fix an ordering  of the three unstable subspaces, and for simplicity we'll choose the obvious one. Then we define
\begin{align} 
& Z_0 = \bar{Z_0} &    &       S_0 = \bar{S_0} \nonumber\\
&Z_1 = \bar{Z_1} - S_0 &&        S_1 = \bar{S_1} - S_0 \nonumber\\ \label{eqn.KNstrata}
&Z_2 = \bar{Z_2} - S_1 - S_0 && S_2 = \bar{S_2} - S_1 - S_0
\end{align}
so the unstable locus is the disjoint union of the three locally-closed subvarieties $S_0, S_1$ and $S_2$.  If we define $\tilde{\tau}_i = \tau_i$ for $i\in[1,2]$, and $\tilde{\tau}_0 = -\tau_0$, then  the subvariety $S_i$ is precisely the attracting subvariety of $Z_i$ under the subgroup $\tilde{\tau}_i$. This data defines a \textit{KN stratification} of the unstable locus, in the sense of \cite{DHL}.\footnote{This is a mild abstraction of a Kirwan--Kempf--Ness stratification. In the latter construction, the ordering on the strata is (partly) determined by a choice of inner product on the Lie algebra of $T'$, but that is irrelevant for our purposes.}

Now we proceed as we did in the rank 1 case (Section~\ref{sect.generaltheoryforwindows}), but for each of the strata simultaneously. So for each stratum $S_i$, we define the numerical invariant
$$\eta_i = \tilde{\tau}_i \mbox{-weight of }(\det N^\vee_{S_i}V)|_{Z_i} $$
and in this example we have $\eta_i=2$, for all $i$, as we've already calculated (Section~\ref{sect.wallcrossinginourexamples}). Make a choice of three integers $k_0, k_1, k_2 \in \Z$, one for each stratum. We define the corresponding `window'
$$\cW(k_0, k_1, k_2) \subset \Db(\cX) $$
to be the full subcategory 
$$\cW(k_0,k_1, k_2) = \genconds{\cE}{\begin{aligned} &\forall i,\mbox{ all homology sheaves of }L\sigma_i^* \cE\mbox{ have $\tilde{\tau}_i$-weights}\\[-5pt] &\mbox{lying in the interval } [k_i, k_i + \eta_i) \end{aligned}} $$
where $\sigma_i : Z_i \into V$ are the inclusion maps. The theorem is that the restriction functor
$$\cW(k_0,k_1, k_2) \to \Db(X^{012}) $$
is an equivalence, for any values of $k_0,k_1,k_2$.

Now let's move to the adjacent phase $X^{021}$, and construct a KN stratification by the same process (again using the obvious ordering on the unstable subspaces). Only the stratum $S_2$ changes. $S_0$ and $S_1$ stay the same, as do the fixed loci $Z_0,Z_1, Z_2$. Also the destabilizing 1-parameter subgroups are now $\tilde{\tau}_0=-\tau_0$, $\tilde{\tau}_1=\tau_1$, and $\tilde{\tau}_2=-\tau_2$. If we pick an integer $l_i$ for each stratum then again have a window
$$\cW^{021}(l_0, l_1, l_2) \subset \Db(\cX)$$
where the weight-restriction condition is now defined with respect to this new KN stratification. The general machinery tells us that this window is equivalent, under restriction, to $\Db(X^{021})$. However, because the $Z_i$ have not changed, it's immediate that
$$\cW^{021}(k_0, k_1, -k_2 -1) = \cW(k_0, k_1, k_2)$$
(the shift in the parameters here is just due to the slight asymmetry in the definition of a window). Consequently, we get some derived equivalences between $X^{012}$ and $X^{021}$, by lifting into one of these windows and then restricting down again. These equivalences are not new, because this is just another way to describe the derived equivalences that we discussed in Section~\ref{sect.wallcrossinginourexamples}.

 To see this, suppose that we don't restrict all the way down to $X^{012}$, but only to the larger open substack $\cY^{012}_{1,2}$. It's immediate from these definitions that any object in the window $\cW(k_0,k_1, k_2) $ restricts to give an object in the window $\cW^{012}_{1,2}(k_2)$, and it follows that the restriction functor
$$\cW(k_0,k_1, k_2) \to \cW^{012}_{1,2}(k_2)\;\subset \Db(\cY^{012}_{1,2})$$
is an equivalence. So our equivalence $T^{012}_{1,2}$, which we defined via the window $\cW^{012}_{1,2}(0)$, can also be defined via any of the windows $\cW(k_0, k_1, 0)$.

From this discussion, it's clear that these windows in $\Db(\cX)$ are candidates for the category $\cV$ required by Lemma~\ref{lem.Vexists}. Unfortunately there is a subtlety which means we cannot apply this machinery immediately. If we move to a more distant phase, more unstable loci will change, and because of the precise definition \eqref{eqn.KNstrata} of the KN strata this means that the $Z_i$ will change (their closures, the $\bar{Z_i}$, remain constant). Consequently, the definitions of the windows for different phases are different, and it is not immediately obvious that there is a single window lifting all of the $\cW^{\sigma}_{i,j}(0)$.\footnote{We are free to change our ordering on the strata for different phases, but this doesn't solve this problem.} To get around this, we use an alternative characterization of these windows.

\begin{prop} \label{prop.W0} Let $\cO(a,b,c)$ (with $a+b+c=0$) denote the equivariant line bundle on $V$ associated to the corresponding character of the torus $T'$, and let
$$ \cV = \left<\;\cO(0,0,0), \,\cO(-1,1,0), \,\cO(-1,0,1)\;\right> \subset \Db(\cX) $$
denote the subcategory generated by these three line bundles. Let $\cW(k_0,k_1,k_2)\subset \Db(\cX)$ be a window defined with respect to the KN stratification for the standard phase as described above. Then
$$\cW(0,0,0) = \cV. $$
\end{prop}
\begin{proof}
If we evaluate any of these three characters on any of the three subgroups $\tau_1, \tau_2$ or $-\tau_0$ then the result is either 0 or 1, so these three line bundles do lie in $\cW(0,0,0)$ (and they are the only line bundles on $\cX$ which do). Furthermore, the definition of windows implies that any element of $\Db(\cX)$ which has a finite resolution by copies of these three line bundles also lies in $\cW(0,0,0)$. This proves that $\cV\subset \cW(0,0,0)$, and hence the restriction functor
\begin{equation}\label{eqn.restrictionfromcV} \cV \to \Db(X^{012})\end{equation}
is fully faithful. We claim that this functor is also essentially surjective. The proposition then follows since the inclusion of $\cV$ into $\cW(0,0,0)$ must be an equivalence.

 Since $V$ is a vector space, any element of $\Db(\cX)$, and hence any element of $\Db(X^{012})$, has a finite resolution by line bundles from the set $\set{\cO(a,b,c)}$. Also, on $X^{012}$ we have three short exact sequences
\begin{align*} 0 \longrightarrow \;\cO(-2,1,1)\; \xrightarrow{\left(\begin{smallmatrix}-a_2&b_0\end{smallmatrix}\right)} \cO(-1,1,0)\oplus \cO(-1,0,1) \xrightarrow{\left(\begin{smallmatrix}b_0\\a_2\end{smallmatrix}\right)} \cO \longrightarrow 0 \\
0 \longrightarrow  \cO(-1,2,-1) \xrightarrow{\left(\begin{smallmatrix}-a_1&b_0\end{smallmatrix}\right)} \cO(-1,1,0)\oplus \cO(0,1,-1) \xrightarrow{\left(\begin{smallmatrix}b_0\\a_1\end{smallmatrix}\right)} \cO \longrightarrow 0 \\
0 \longrightarrow  \cO(-1,-1,2) \xrightarrow{\left(\begin{smallmatrix}-a_2&b_1\end{smallmatrix}\right)} \cO(0,-1,1)\oplus \cO(-1,0,1) \xrightarrow{\left(\begin{smallmatrix}b_1\\a_2\end{smallmatrix}\right)} \cO \longrightarrow 0 
 \end{align*}
associated to the three strata of the unstable locus. By repeatedly using twists of these sequences we can resolve any of the line bundles $\cO(a,b,c)$ on $X^{012}$ in terms of the three line bundles that generate $\cV$. Consequently any element of $\Db(X^{012})$ has a finite resolution in terms of these three line bundles, and so \eqref{eqn.restrictionfromcV} is essentially surjective as claimed.
\end{proof}

\begin{rem}\label{rem.delicate_features} There are some delicate features of this argument that are worth highlighting. The fact that the values of the three parameters $k_0=k_1=k_2=0$ are correctly aligned is crucial: for most choices of these integers the window $\cW(k_0,k_1, k_2)$ contains no line bundles at all. Also the precise action of $\tilde{T}$ on $\tilde{V}$ is important: for a general example of a torus action on a vector space there won't be \textit{any} choice of integers on the strata such that the associated window is generated by line bundles.
\end{rem}

Now we can show that this definition of $\cV$ suffices to produce all six of our wall-crossing functors.

\begin{proof}[Proof of Lemma~\ref{lem.Vexists}.]
For each phase $X^\sigma$, we define $k^\sigma_i = 0$ or $k^\sigma_i = -1$ according to whether $\tau_i$ or $-\tau_i$ is the destabilizing subgroup in this phase, i.e. according to whether $\theta(\tau_i)$ is positive or negative.

Now let $\sigma$ and $\sigma'=\sigma\cdot(ij)$ be two consecutive phases in the hexagon \eqref{eqn.hexagon}. For the phase $X^\sigma$, pick any ordering on the unstable subspaces, and then define a window $\cW^\sigma(k^\sigma_0, k^\sigma_1, k^\sigma_2)\subset \Db(\cX)$ using the corresponding KN stratification. Arguing as above, the restriction functor
$$\cW^\sigma(k^\sigma_0, k^\sigma_1, k^\sigma_2) \to \Db(\cY^\sigma_{i,j})$$
is an embedding with image $\cW^\sigma_{i,j}(0)$. On the other hand, the proof of Proposition~\ref{prop.W0} adapts immediately to show that 
$$\cW^\sigma(k^\sigma_0, k^\sigma_1, k^\sigma_2) = \cV,$$
and so our claim is proved.
\end{proof}

\begin{rem}\label{rem.philosophy_and_poss_generalizatn} As we move around between phases, the definitions of our windows in $\Db(\cX)$ change, but this result says that the actual windows do not change (for appropriate choices of the parameters). We've proved this using the fact that the relevant windows are generated by line bundles, but this is a very special condition and we believe that this result should actually hold in more generality. If true, this would imply analogous relations between autoequivalences for other GLSM examples.

In fact this is the main reason that we've included some discussion of the general machinery of windows on $\cX$. If one wants to avoid it then it's reasonably straightforward to prove Lemma~\ref{lem.Vexists} directly from our definition of $\cV$.
\end{rem}

\subsection{The poset of groupoid actions}
\label{sect.posetGroupoidActions}

When discussing functors between derived categories of schemes, one doesn't normally work in $\Catni$ but rather in the subcategory $\mathbf{FM}$ where the morphisms between $\Db(Y)$ and $\Db(Z)$ are (isomorphism classes of) Fourier--Mukai kernels, that is objects of $\Db(Y\times Z)$. This subcategory is better-behaved in a number of ways. In our situation, we're actually working in a slightly more specialized subcategory.

 Recall that each of our varieties $X_\Gamma^\sigma$ is equipped with a fibration 
$$\mu: X_\Gamma^\sigma \to \C^s,$$
where $(s+1)$ is the number of parts of $\Gamma$. Define $\mathbf{FM}_{\C^s}$ to be the category whose objects are schemes flat over $\C^s$, and whose morphisms are relative Fourier--Mukai kernels, i.e. objects of 
$\Db(Y\times_{\C^s} Z). $
Now let $\tilde{\Gamma}$ be a coarsening of $\Gamma$, with $\tilde{\Gamma}$ having $(\tilde{s}+1)$ parts. Associated to this data is a linear inclusion
$$ j: \C^{\tilde{s}} \into \C^s, $$
and each subvariety $X_{\tilde{\Gamma}}^\sigma$ is the fibre product of the corresponding $X_\Gamma^\sigma$ over this subspace.

\begin{prop}\label{prop.restrictionofkernels} There is a functor
\begin{align*}\Res: \mathbf{FM}_{\C^s} & \to \mathbf{FM}_{\C^{\tilde{s}}} \\
Y & \mapsto Y\times_{\C^s} \C^{\tilde{s}},\end{align*}
acting on morphisms by the derived restriction functor
\begin{align*} \Db(Y\times_{\C^s} Z) & \xrightarrow{j^*} \Db( \Res(Y\times_{\C^s} Z )) \\
& = \Db\left( \Res(Y) \times_{\C^{\tilde{s}}} \Res(Z) \right ). \end{align*}
\end{prop}

\begin{proof}Indeed we may replace the linear inclusion
$ j:\C^{\tilde{s}} \into \C^s $ with any morphism of schemes $f: T \to S$, so that we define $\Res(Y) :=Y_T= Y \times_{S} T$, and use derived pullback $f^*$ for the morphisms. Then taking schemes $X$, $Y$, and $Z$, all flat over $S$, and Fourier--Mukai kernels $$P \in \Db( X \times_S Y), \qquad Q \in \Db( Y \times_S Z),$$ we verify that $\Res$ is indeed a functor: composition of morphisms in $\mathbf{FM}_S$ is by convolution of kernels, so we require $$\Res(P \ast Q) \cong \Res(P) \ast \Res(Q).$$ We have a diagram
$$\begin{tikzcd} X_T \times_T Y_T \times_T Z_T \dar{p'_{13}} \drar[pos=0.3]{f'} \\
X_T \times_T \drar[pos=0.25]{f} Z_T & X \times_S Y \times_S Z \dlar[swap,pos=0.25]{p_{12}} \dar[pos=0.3]{p_{13}} \drar[pos=0.3]{p_{23}} \\
X \times_S Y & X \times_S Z & Y \times_S Z \end{tikzcd}$$
and by definition
\begin{align*}\Res(P \ast Q) & := \Res(p_{13*} (p^*_{12} P \otimes p^*_{23} Q)) \\
 & := f^* p_{13*} (p^*_{12} P \otimes p^*_{23} Q) \\ 
 & \cong p'_{13*} f'^* (p^*_{12} P \otimes p^*_{23} Q) \\
&  \cong p'_{13*} (f'^* p^*_{12} P \otimes f'^* p^*_{23} Q) \end{align*}
  where the first isomorphism is base change using flatness of $p_{13}$ \cite[Theorem 3.10.3]{LipmanNotes}. The proof is finished using the commutativity of the following two squares
$$\begin{tikzcd}
X_T \times_T Y_T \dar{f} & X_T \times_T Y_T \times_T Z_T \lar[swap]{p'_{12}} \dar{f'} \rar{p'_{23}} & Y_T \times_T Z_T \dar{f} \\
X \times_S Y  & X \times_S Y \times_S Z \lar[swap]{p_{12}} \rar{p_{23}} & Y \times_S Z
\end{tikzcd}$$
and the definition of  $\Res(P) \ast \Res(Q)$.\end{proof}

Recall that in Section \ref{sect.heuristicsforthegenerators} we defined derived equivalences 
$$T^{\sigma S_\Gamma}_{i,j}: D^b(X^\sigma_\Gamma) \to D^b(X^{\sigma\cdot (ij)}_\Gamma)$$ 
corresponding to generators of the groupoid $\cG_\Gamma$.

\begin{prop} Each of our functors $T^{\sigma S_\Gamma}_{i,j}$ is actually a morphism in $\mathbf{FM}_{\C^s}$, i.e. their kernels are well-defined relative to the fibration $\mu$.\end{prop}
\begin{proof}
 When $T^{\sigma S_\Gamma}_{i,j}$ corresponds to a flop between two distinct phases (case \hyperref[casea]{(a)})  then this is clear from the geometric construction of the functor, because the common birational roof of the two phases is evidently defined relative to $\mu$. When $T^{\sigma S_\Gamma}_{i,j}$ is a family spherical twist acting on a single phase (case \hyperref[caseb]{(b)}) then one can either verify the claim directly by inspecting the kernel for a family spherical twist, or just deduce it from case (a) using the next proposition.\end{proof}

\begin{prop}\label{prop.groupoid_action_intertwinement} Choose a partition $\Gamma$, and let $\tilde{\Gamma}$ be a coarsening of $\Gamma$. Pick $\sigma, i, j$ such that the equivalence $T^\sigma_{i,j}$ is defined (and hence so are $T^{\sigma S_\Gamma }_{i,j}$ and $T^{\sigma S_{\tilde{\Gamma}}}_{i,j}$).  Then
\begin{equation}\label{eqn.kernels_restrict} \Res \left(T^{\sigma S_\Gamma }_{i,j}\right) = T^{\sigma S_{\tilde{\Gamma}}}_{i,j}, \end{equation}
and in particular $T^{\sigma S_\Gamma }_{i,j}$ and $T^{\sigma S_{\tilde{\Gamma}}}_{i,j}$ are intertwined by pullback along the inclusions.\end{prop}
\begin{proof}
Let $\sigma' = \sigma\cdot (ij)$. We have four varieties as follows:
$$\begin{tikzcd} X^\sigma_\Gamma \rar[<->,dashed] & X^{\sigma'}_\Gamma  \\
X^{\sigma\phantom{'}}_{\tilde{\Gamma}} \uar[hook]{j} \rar[<->,dashed] & X^{\sigma'}_{\tilde{\Gamma}} \uar[hook,swap]{j} \end{tikzcd} $$
The last part is the statement that $$ j^* \circ T^{\sigma S_{\Gamma}}_{i,j} \cong T^{\sigma S_{\tilde{\Gamma}} }_{i,j} \circ j^*  ,$$ and this follows from \eqref{eqn.kernels_restrict} using standard base change results \cite[Proposition 6.1]{FMNahm}.

To prove \eqref{eqn.kernels_restrict}, suppose firstly that $i$ and $j$ are in different parts of $\tilde{\Gamma}$, and so also in different parts of $\Gamma$. Then both birational maps are flops of a trivial family of  rational $(-1,-1)$-curves (as described in Section~\ref{sect.familiesofrationalcurves}), and the functors $T^{\sigma\Gamma}_{i,j}$ and $T^{\sigma\tilde{\Gamma}}_{i,j}$ are the associated birational equivalences. Away from the flopping families both functors are trivial, so we only need to check the equality inside our very simple Zariski neighbourhoods of the flopping families. In these neighbourhoods, we're considering a trivial family of standard flops over $\C^{s-1}$, and then just restricting to a subspace $\C^{\tilde{s}-1}\subset \C^{s-1}$, so the equality of the functors is clear.
 
Now suppose that $i$ and $j$ are in the same part of $\Gamma$, and so also in the same part of $\tilde{\Gamma}$. A similar argument applies, except that now the birational maps are identities and the functors are family spherical twists. 

The final case is that $i$ and $j$ are in different parts of $\Gamma$, but in the same part of $\tilde{\Gamma}$. By the same argument we can reduce to our simple Zariski neighbourhoods, and then we're considering a trivial family of copies of the geometry considered in Proposition~\ref{prop.flopinducestwist}, so the argument from there suffices.
\end{proof}

Combining the above three propositions with Proposition~\ref{prop.braidrelationsonversal}, we can formally deduce the whole of Theorem~\ref{thm.mainthm}.

\begin{proof}[Proof of Theorem~\ref{thm.mainthm}]
As before we let $\cG_\Gamma^{\free}$ be the free groupoid generated by all the arrows $t^{\sigma S_\Gamma}_{i,j}$, so $\cG_\Gamma$ is a the quotient of $\cG_\Gamma^{\free}$ by the braid relations. If $\tilde{\Gamma}$ is a coarsening of $\Gamma$, then the above three propositions give us a commutative square as follows:
$$\begin{tikzcd} \cG_\Gamma^{\free} \rar{T_\Gamma} \dar & \mathbf{FM}_{\C^s} \dar{\Res} \\
\cG_{\tilde{\Gamma}}^{\free} \rar{T_{\tilde{\Gamma}}} & \mathbf{FM}_{\C^{\tilde{s}}} \end{tikzcd} $$
If we set $\Gamma = \Gamma_{\fin}$, then Proposition~\ref{prop.braidrelationsonversal} says that $T_\Gamma$ factors through $\cG_\Gamma$. Therefore $T_{\tilde{\Gamma}}$ must factor through $\cG_{\tilde{\Gamma}}$, and this holds for all partitions $\tilde{\Gamma}$.
\end{proof}

Informally, the above proof just says that since the braid relations hold between the equivalences on the ambient space, they must continue to hold when we restrict these equivalences to subvarieties.

Finally, we use the faithfulness result of Seidel--Thomas to deduce Theorem~\ref{thm.faithfulness}.

\begin{proof}[Proof of Theorem~\ref{thm.faithfulness}]
Let $\Gamma$ be any partition, and consider coarsening it to $\Gamma_{\crs}$. We have a commutative square as follows:
$$\begin{tikzcd}[column sep=16pt]  & \cG_\Gamma \arrow{rr}{T_\Gamma} \dar & & \mathbf{FM}_{\C^s} \dar{\Res} \\
B_{k+1} \arrow[-,transform canvas={yshift=+\sepForEquals}]{r} \arrow[-,transform canvas={yshift=-\sepForEquals}]{r} & \cG_{\crs} \arrow{rr}{T_{\crs}} & & \mathbf{FM}_{\pt} & \phantom{B_{k+1}}
 \end{tikzcd} $$
The functor $T_{\crs}$ is precisely the braid group action on the derived category of the surface $Y_0 = X_{\crs}$ considered in \cite{ST}, and it is proved there that the action is faithful. Since the functor $\cG_\Gamma \to B_{k+1}$ is also faithful, we must have that $T_\Gamma$ is faithful.
\end{proof}

% ----------------------------------------------------------------
\newpage
\appendix
\section{List of notations}
\label{list of notations}

\begin{longtable}{p{1.8cm} p{11cm} p{3cm}}
$T \curvearrowright V$ & torus $T$ acting on vector space $V$ & \\
$\theta$ & character of $T$ \\
$X_\theta$ & GIT quotient $V \sslash_\theta T$ \\
$Y_0$ & small resolution of $A_k$ surface singularity \\
$\cX$ & Artin quotient stack $\quotstack{V}{T}$ \\
$\cF$ & Fayet--Iliopoulos (FI) parameter space \\
$\Db$ & bounded derived category of coherent sheaves\\ 
$B_{k+1}$ & braid group on $k+1$ strands, labelled $[0,k]$ \\
$\tilde{B}_{k+1}$ & affine braid group associated to affine Dynkin diagram $\tilde{A}_k$ \\
$\cW$ & window, in the bounded derived category of a stack & \S\ref{sect.A1case}, \S\ref{sect.derivedequivalencesfromwalls} \\
$\psi$ & window equivalence, corresponding to wall-crossing & \eqref{eqn.3fold_window_equiv}, \S\ref{sect.derivedequivalencesfromwalls} \\
$T_{\mathcal{S}}$ & spherical twist, around spherical object (or functor) $\mathcal{S}$ & \eqref{eqn.ST_twist}, \eqref{eqn.ST_family_twist} \\
$a_i$, $b_i$ & co-ordinates on $V$ for clockwise, and anti-clockwise, arrows & \S\ref{section.statement} \\
$\mu_i$ & complex moment maps & \eqref{eqn.moment_maps} \\
$\Gamma$ & partition of $[0,k]$, with pieces indexed by $[0,s]$ & \S\ref{section.construction} \\
$S_\Gamma$ & Young subgroup of $S_{k+1}$, preserving partition $\Gamma$ & \eqref{eqn.mixed_braid_defn} \\
$B_\Gamma$ & mixed braid group: subgroup of $B_{k+1}$, preserving partition $\Gamma$ & \eqref{eqn.mixed_braid_defn} \\
$Q$ & toric data: matrix of weights & \S\ref{sect.reps_free_quiver} \\
$\Delta^k$ & $k$-simplex & \S\ref{sect.reps_free_quiver} \\
$\Pi$ & polytope & \S\ref{sect.reps_free_quiver} \\
$\polyTopVec{i}, \polyBaseVec{i}$ & vectors in polytope, generating rays in the fan for $X$ & \S\ref{sect.reps_free_quiver} \\ 
$\vartheta$ & lift of character $\theta$ & \eqref{thetaform} \\
$X_\Gamma$ & moment map subvariety of $X$ & \S\ref{sect.subvarieties} \\
$\gamma$ & function $[0,k] \to [0,s]$ encoding partition $\Gamma$ & \S\ref{sect.subvarieties} \\
$N_t$ & size of $t^{\text{th}}$ piece of partition $\Gamma$ & \S\ref{notn.coords} \\
$\polyVec{t}{\delta}$ & vectors in polytope, generating rays in the fan for $X_\Gamma$ & \eqref{eq.polyVecDef} \\
$\tilde{V} \sslash \tilde{T}$ & GIT quotient giving subvariety $X_\Gamma$ & \S\ref{sect.subvarietyGITdescrip} \\
$X^\sigma_\Gamma$ & phase corresponding to partition $\Gamma$, permutation $\sigma$ & \S\ref{sect.subvarietyGITdescrip} \\
$S_i$, $\tilde{S}_i$ & flopping subvarieties of $X$, $X_\Gamma$ & \eqref{eq.twistSubvarieties}, \eqref{eqn.twistSubvarietiesNonVersal} \\
$\cF_\Gamma$ & FI parameter space for variety $X_\Gamma$ & \S\ref{sect.resultsForExamples} \\
$\zeta_t$ & FI parameter & \S\ref{sect.resultsForExamples} \\ 
$\Gamma_{\fin}$, $\Gamma_{\crs}$ & finest, and coarsest, partitions $\Gamma$ & \S\ref{sect.resultsForExamples} \\
$\cG_\Gamma$ & a path groupoid, associated to $\cF_\Gamma$ & \S\ref{sect.heuristicsforthegenerators} \\
$t^{\sigma S_\Gamma}_{i,j}$ & a morphism in $\cG_\Gamma$, crossing adjacent strands $i$ and $j$ & \S\ref{sect.heuristicsforthegenerators} \\
$T^{\sigma S_\Gamma}_{i,j}$ & an equivalence, with source $\Db(X^\sigma_\Gamma)$ & \S\ref{sect.heuristicsforthegenerators} \\
$\Catni$ & category of categories, with functors up to isomorphism & \S\ref{sect.heuristicsforthegenerators} \\
$\tau_i$ & $1$-parameter subgroup of $T$ & \S\ref{sect.formalStructureOfProof} \\
$S_i$ & unstable stratum & \S\ref{sect.formalStructureOfProof}\\
$Z_i$ & fixed loci of unstable stratum & \S\ref{sect.formalStructureOfProof}\\
$\eta_i$ & window width & \S\ref{sect.formalStructureOfProof}\\
$\mathbf{FM}$, $\mathbf{FM}_{B}$ & category of Fourier--Mukai kernels between varieties (relative to $B$) & \S\ref{sect.posetGroupoidActions} \\
\end{longtable}
% ----------------------------------------------------------------

\end{document}